\documentclass{amsart}

\usepackage{amsfonts,amssymb,amscd,amsmath,enumerate,verbatim,newlfont,calc}

\usepackage{amsfonts}

\newtheorem{theorem}{Theorem}[section]

\usepackage{enumerate}
\usepackage{graphicx}
\usepackage{epstopdf}
\usepackage{mathtools}
\usepackage{amsmath}
\usepackage{subfigure}
\usepackage{epsfig}
\usepackage{amsmath,amssymb}

\begin{document}

\title[Synchronization of Chaos in Ecology
~~~~~ ]{ Using Prey Abundance to Synchronize Two Chaotic GLV Models   \footnote { \\ \it Journal of Dynamical Systems $\&$ Geometric Theories \\
\\
\copyright Taru Publications}}
\author[Shubhangi ]{ {\bf {Shubhangi Dwivedi}}{
\\e-mail: shubhangi.dwivedi176@gmail.com }\\}
\author[ Nitu Kumari]{ {\bf {Nitu Kumari}}{
\\e-mail:  nitu@iitmandi.ac.in }
\\{School of Mathematical and Statistical Sciences, Kamand, Indian Institute of Technology, Mandi-Himachal Pradesh, 175005, India}
\\}
\author[R.K.Upadhyay]{ {\bf { Ranjit Kumar Upadhyay}}{
\\{Department of Applied Mathematics, Indian Institute of Technology (ISM), Dhanbad-Jharkhand, 826004 , India}
\\e-mail: ranjit.chaos@gmail.com }}

\maketitle
\begin{center}
\tiny ( Accepted: 02 November 2021 )
\end{center}

\begin{abstract}
\tiny
The concept of superfluous prey, or an excess of prey in certain areas within a patchy ecosystem, has significant implications for the synchronization of the predator population. These areas, known as "hotspots," have a higher density of prey compared to other areas and attract a higher concentration of predators. As a result, the predator population becomes more stable and predictable, as they are less likely to migrate to other areas in search of food. This phenomenon can have important consequences for both the predators and their prey, as well as the overall functioning of the ecosystem. This work investigates the synchronization between two chaotic food webs using the generalized Lotka-Volterra (GLV) model consisting of one prey and two predator populations. We, first, examine the impact of three functional responses (linear, Holling type II, and Holling type III) on system dynamics For the study, we consider the model with a linear functional response consisting of chaotic oscillations and apply controllers to stabilize its unstable fixed points. This research contributes to the understanding of how to apply chaotic ecological models to predict the population of competing species in one habitat using information about similar populations in another system. To do this, we configure a drive-response system where prey acts as the driving variable and both predators depend only on the prey. We use active and adaptive control methods to synchronize two coupled GLV models and verify the analytical results through numerical simulations.
\end{abstract}

\noindent {\bf AMS Classification:} {--- }
\vspace{1cm}
\maketitle

\section{Introduction}\label{sec1}
In ancient philosophy and mythology, the word chaos meant the disordered state of unformed matter supposed to have existed before the ordered universe. In the course of its journey to truth, science has met with startling phenomena called chaos. Since the 1960s, with the discovery of chaotic systems, chaos has set a nonlinear dynamics research boom. Chaos theory is attributed to the work of Edward Lorenz. His 1963 paper, ``Deterministic Non-periodic Flow" \cite{lorenz1963deterministic}, is credited for laying the foundation for chaos theory. Hunt and Ott \cite{hunt2015defining} reviewed the problem and proposed a computationally feasible entropy-based  good definition of chaos. They define chaos as `` the existence of positive Expansion Entropy (EE) (equal to topological entropy for infinitely differentiable maps) on a given restraining region (bounded positive volume subset),'' which confirms both the notions (COS as well as OS). Since EE enjoys the properties of simplicity, computability, and generality, so they call it a `` good '' definition of chaos. Chaotic systems are sensitive to initial conditions, topologically mixing and with dense periodic orbits \cite{solari1996nonlinear}, \cite{strogatz2018nonlinear}. Because of slightest difference, chaotic dynamical systems can lead to entirely different trajectories.  The main characteristic of chaos is that the system does not repeat its past behaviour. Mathematically, chaotic dynamical systems are classified as non-linear dynamical systems having at least one positive Lyapunov exponent\cite{janaki2003lyapunov}.  
\\
Chaos theory is not just about chaos - it has two sides to it: chaos control and chaos synchronization. The study of chaos control and understanding chaotic model behaviour has gained significant interest, with applications in various fields such as secure communications, biology, neural networks, finance, and more. Whereas chaos synchronization refers to the alignment in time of different chaotic processes. At first glance, chaotic systems may seem to defy synchronization, but it has been observed and studied in various contexts.\\
In $1665$, Dutch physicist Huygens observed the adjustment of rhythms via a coupling. He noticed that pendulum clocks suspended from the same beam would slowly adjust their phases. In $1984$, Kuramoto set theory for the onset of sync, and  Pecora and Carroll \cite{pecora1990synchronization} reviewed the area of synchronization in chaotic systems and presented a more geometric view using synchronization manifold. In $1999$, Blasius explained the theoretical analysis of seasonally synchronized chaotic population cycles  \cite{blasius1999complex}. All these contributions help the researcher think that synchronization is an essential phenomenon in physical and biological systems. In literature, this phenomenon has been nominated  by various types, such as phase locking, frequency pulling, generalized synchrony, and complete locking, depending on the degree or type of synchronization. \\
Synchronization of chaotic systems can be achieved by configuring drive and response systems, with the goal of using the output of the drive system to control the response system so that the output tracks the drive system asymptotically. However, creating identical chaotic synchronized ecological systems is questionable as it has a potential threat to biodiversity. As discovered in Pecora and Carroll's pioneering work, another way of achieving complete synchronization between two systems is what is now called the technique of complete replacement. The complete replacement synchronization is helpful in a network of patchy ecosystems as it can help in identifying the underlying mechanism that derive the group co-ordination in the present of severely chaotic oscillations. In population biology, the chaotic dynamics  may synchronize if populations are coupled through environmental or biological interactions.\\
From the ecological aspect, it is crucial to figure out the ecologically feasible coupling scheme that guarantee the permanence and global attractiveness of all species in a multi-patch ecosystem. Upadhyay and Rai \cite{upadhyay2009complex} have demonstrated that the two non-identical chaotic ecological systems having different kinds of top-predators can be synchronized using an algorithm proposed by Lu and Cao\cite{lu2005adaptive}. The idea of this approach is that it takes care of the uncertainties involved in the parameter estimation. There are many methods in control theory for synchronizing chaotic systems, including the Adaptive Control Method, Back-stepping Method, Active Control Method, Time-Delay feedback approach, and others.\\
Among above-listed methods, the linear active control technique for chaos synchronization is popular and effective for synchronizing identical and non-identical chaotic systems. In this work, we will use the active and adaptive control methods for achieving synchronization of chaos within either identical or non-identical systems. The Active control method was first used for chaos synchronization by E.W. Bai and K.E. Lonngren \cite{bai1997synchronization}, \cite{njah2009synchronization}. In this method, non-linear controllers are designed based on the Lyapunov stability theory to achieve synchronization in coupled systems using the known parameters of the drive and response systems.
\\
Since we will be dealing with an ecological model, the system's parameters cannot be precisely known. The adaptive control is one of the popular technique for controlling and synchronizing non-linear systems with uncertain parameters \cite{chen2002synchronization}. The method allows the model to adapt data assimilation along the way that may be useful for predicting the real system's future behaviour \cite{dai2001chaos}, \cite{park2005adaptive}. In this method, control law and parameter update law are designed in such a way that the chaotic response system to behave like chaotic drive systems. This scheme maintains the consistent performance of a system in the presence of uncertainty, variations in parameters. Consequently, asymptotically global synchronization control of the chaotic system guarantees to converge the error dynamics to the equilibrium point.
\\
In this article, we investigate the three-dimensional chaotic generalized Lotka-Volterra system, a more general model than the competitive predator-prey examples of Lotka-Volterra types. We examine the properties of the model, including equilibrium analysis, dissipative properties, the maximum Lyapunov exponent, and slow manifold analysis. We use linear feedback control to stabilize the model at its equilibrium points. Our goal is to synchronize two identical GLV models with the same drive variable and different initial conditions using complete replacement synchronization in an ecological context. Prey population is assumed to be in abundance in such a way that predators from nearby patches also feed upon it. To maintain synchronization indefinitely with only small adjustments within a two-patch system, we design the active control law(when system parameters are known) and adaptive control law (when system parameters are unknown) mathematically and validate the analytical results through numerical simulation.

\section{\label{sec2} The Model}

The generalized Lotka-Volterra equations are autonomous and deterministic. The dynamics of the model in a more generalized form are defined as

$$\dot{x}_{i}(t) = {x}_{i}(t)(r_{i} + {f}_{i}({x}_{1},{x}_{2},\dots,{x}_{n})),$$
with initial condition 
$$ x_{i}(0) = {x}_{(i,0)}, ~ for~ {i} \in \{ 1, 2,\dots, n\}.$$
where ${n}$ represents the number of species, ${x}_{i}(t)$ is the size of population ${i}$,  $\dot{x}_{i}(t)$ is the time derivative of species ${i}$, ${t}$ is the time variable, ${x}_{(i,0)}$ is the initial population of species ${i}$, ${r}_{i}$ is the self-growth of species ${i}$, and ${f}_{i}(x_{1},x_{2},x_{3}, \dots,{x}_{n} )$ is the nonlinear multi-variable function with intra and inter-species competition terms for each ${i}$. Although the populations are usually measured in integer numbers, ${x}_{i}(t)$ is real for each ${i}$ and can be interpreted as density, biomass or some other measure which correlates with the number of species. Let ${R}^{n}$ denotes the Euclidean space and the function ${f}_{i}$ is a  continuous, smooth and real-valued function for each ${i}$. We make the following assumptions on ${f}_{i}$ for ${i} \in \{ 1, 2,\dots, n \}$ \cite{rao1999global}.\\
\begin{enumerate}[(i)] 
    \item  ${f}_{i}$, for each ${i}\in \{ 1, 2,\dots, n \}$ is bounded on a domain ${D} \subset {R}^{n}$.
    \item There exist constant ${K}_{i}>0$ for each ${i}\in \{ 1, 2,\dots,n \}$ such that 
 $$||{f}_{i}(X)-{f}_{i}(Y)|| \leq {K}_{i}~||X-Y|| \quad \forall \quad {X},{Y} \in {D} \subset {R}^{n}.$$
 \end{enumerate}

 \subsection{\label{sec2.1}Model with linear functional response}
The generalized Lotka Volterra (GLV) model and its variant have been studied by many authors \cite{ackleh1900extinction},\cite{malcai2002theoretical},\cite{elsadany2018dynamical}. Our main focus is on three-dimensional GLV chaotic system, which has been devised by Samardzija and Greller in $1988$ \cite{samardzija1988explosive}. We assume that the vector field for the model holds the above mentioned properties and takes the following form 
\begin{equation} 
\begin{split}
{\dot{{x}_{1}}}  & = {x}_{1}(1-{x}_{2}+{r}{x}_{1}-{p}{x}_{3}{x}_{1}),\\
   {\dot{{x}_{2}}}  & =  {x}_{2}(-1 +{x}_{1}), \\ 
   {\dot{{x}_{3}}}   & ={x}_{3}(-{q}+{p}{x}_{1}^{2}). 
\end{split}
\end{equation}

where ${x}_{1},~{x}_{2},{x}_{3}$ are the state variables representing prey, middle predator and top predator populations respectively. Here ${p}, {q}, {r}$ are positive parameters. Authors \cite{elsadany2018dynamical} found the system chaotic in particular parametric range $ p= 2.9851, q= 3, r=2$ and shown interesting complex dynamical behaviour. For this set of parameter values, the orbit of all three states for two different initial conditions ($(1.0023,1.0589,0.6503)$ and $(1.0023+10^{-3}, 1.0589+10^{-3}, 0.6503+10^{-3})$) has sensitive dependence on initial conditions (SDIC) which is displayed in figure \ref{fig1}. Figure \ref{fig1} characterizes the SDIC in the system where the trajectories with initial condition $(1.0023+10^{-3}, 1.0589+10^{-3}, 0.6503+10^{-3})$) dominate over trajectories with initial condition ($(1.0023,1.0589,0.6503)$ in long run.
\begin{figure}[ht]
\centering
\begin{tabular}{c}
 \subfigure[]{\includegraphics[width=10cm,height=4cm]{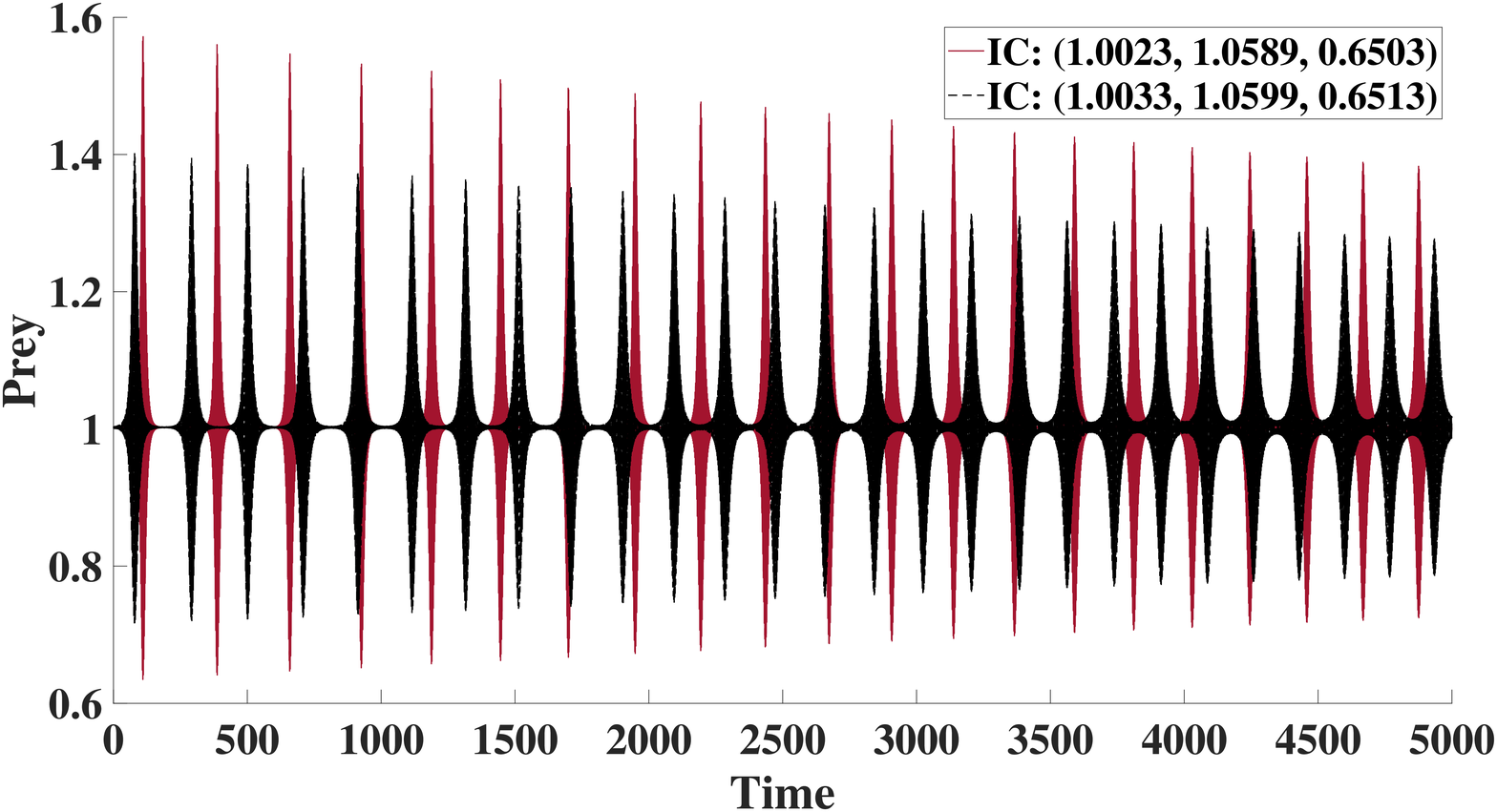}}\\
  \subfigure[]{\includegraphics[width=10cm,height=4cm]{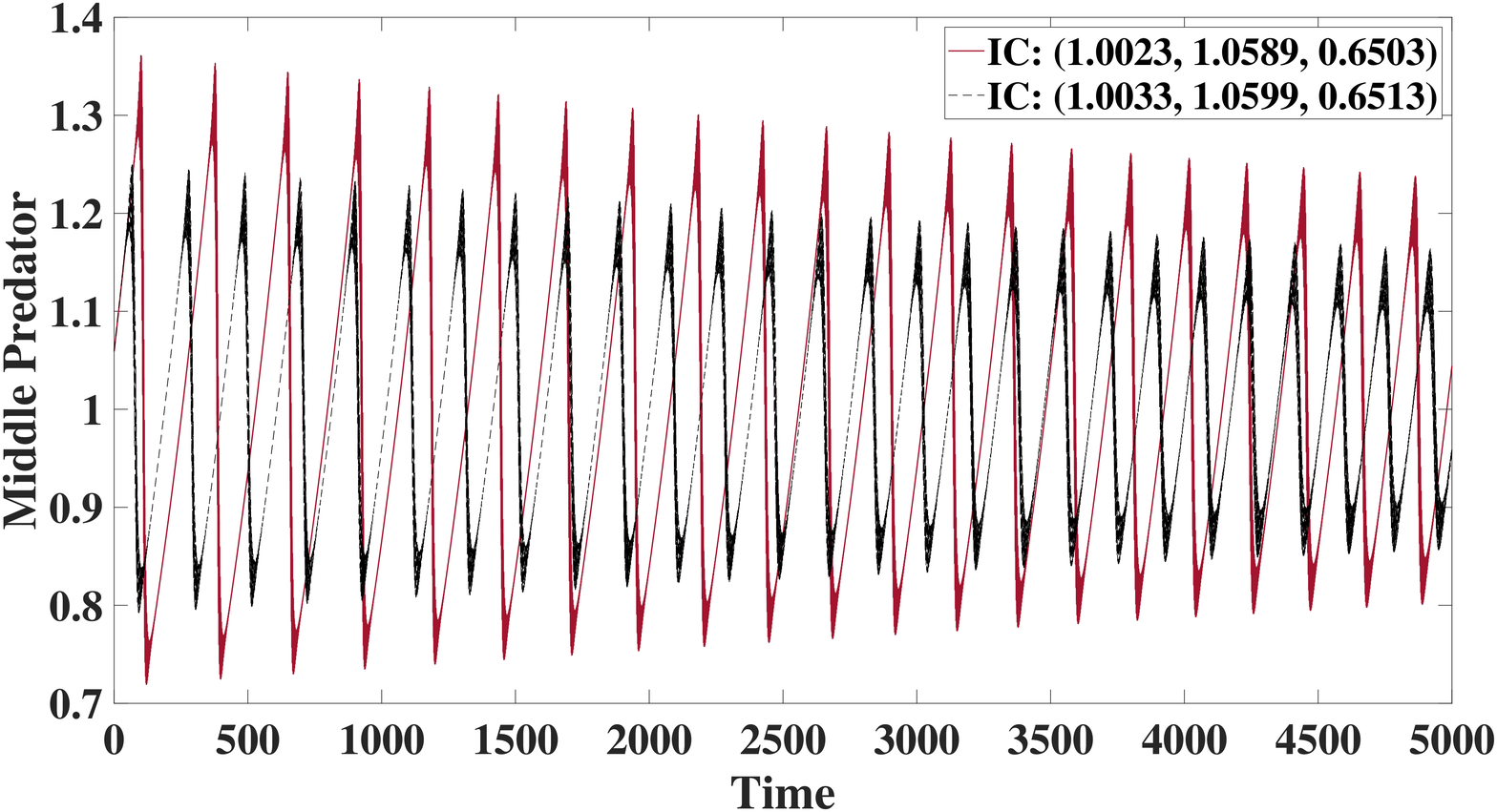}} \\
  \subfigure[]{\includegraphics[width=10cm,height=4cm]{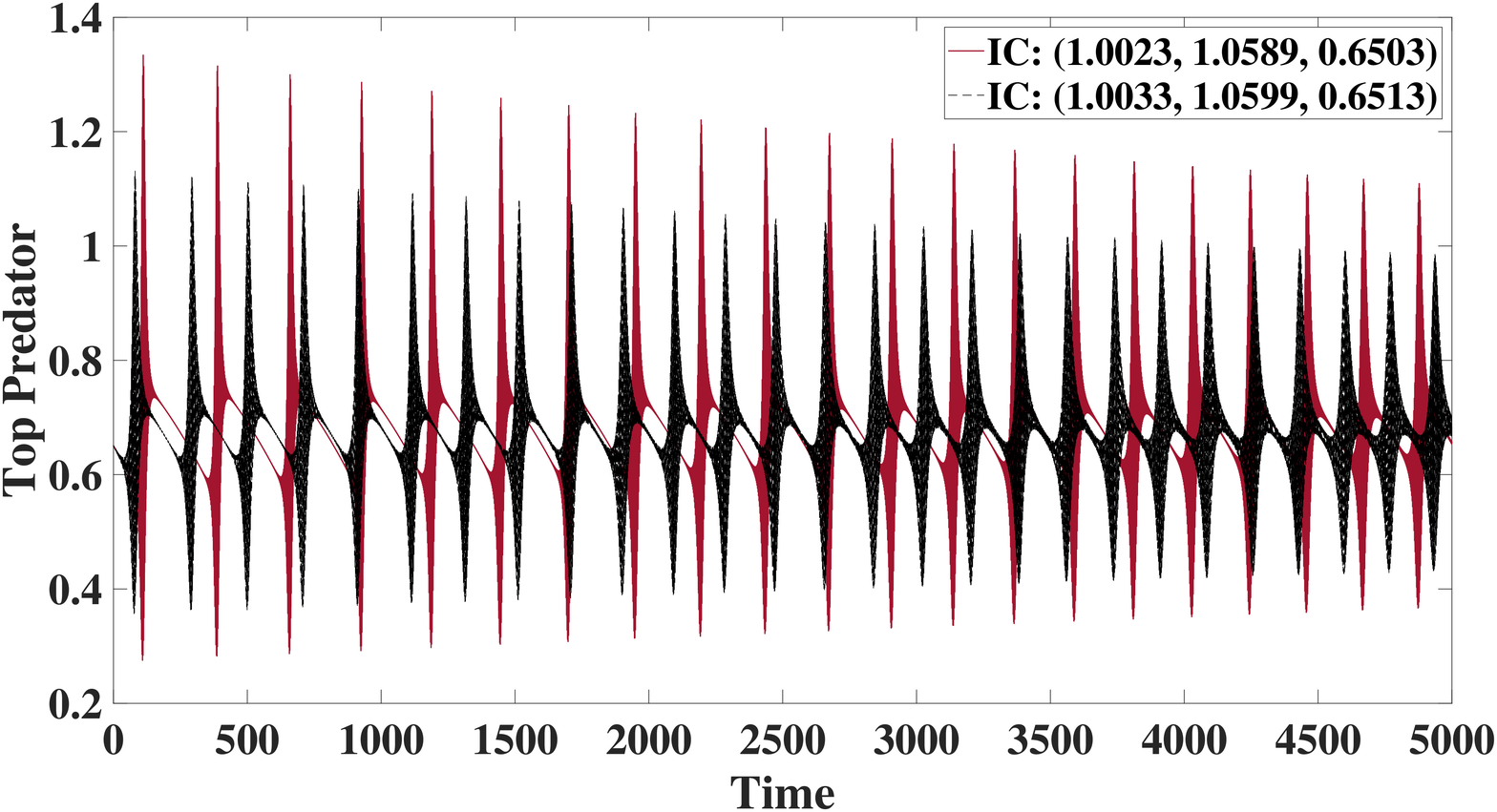}}   
  \label{fig1} 
\end{tabular}
\caption{  (a), (b) and (c):  Time series of ${x}_{1},{x}_{2}$, and ${x}_{3}$ for two nearby initial conditions ($(1.0023,1.0589,0.6503)$ and $(1.0023+10^{-3}, 1.0589+10^{-3}, 0.6503+10^{-3})$).}
\label{fig:1}  
\end{figure}
\\
Further, we display the dynamics of GLV system for three different set of parameters to characterize its parameter- sensitivity. Different sets of parameters involved in the system are taken as 
$$ (p, q, r) \in \{ (2.0451, 2.129,  2), (2.9851, 2.99, 2.1), (2.98098, 2.9799, 2)\}.$$
For simulation, we fix the initial condition at ${x}_{1}(0) = 1.0023,~{x}_{2}(0)= 1.0589,~{x}_{3}(0)= 0.6503$ for the GLV system. Figures \ref{fig:2}, \ref{fig:21} and \ref{fig:22} show  three dimensional attractor  and two dimensional projections of the system on $({x}_{1}, {x}_{2}), ({x}_{1},{x}_{3})$ and $({x}_{2}, {x}_{3})$ planes. From these figures, it can be inferred that the sensitivity of the system on parameters can help in restoring the hidden order out of its chaotic dynamics.
\begin{figure}[ht]
\centering
\begin{tabular}{cc}
 \subfigure[]{\includegraphics[width=6.5cm,height=3cm]{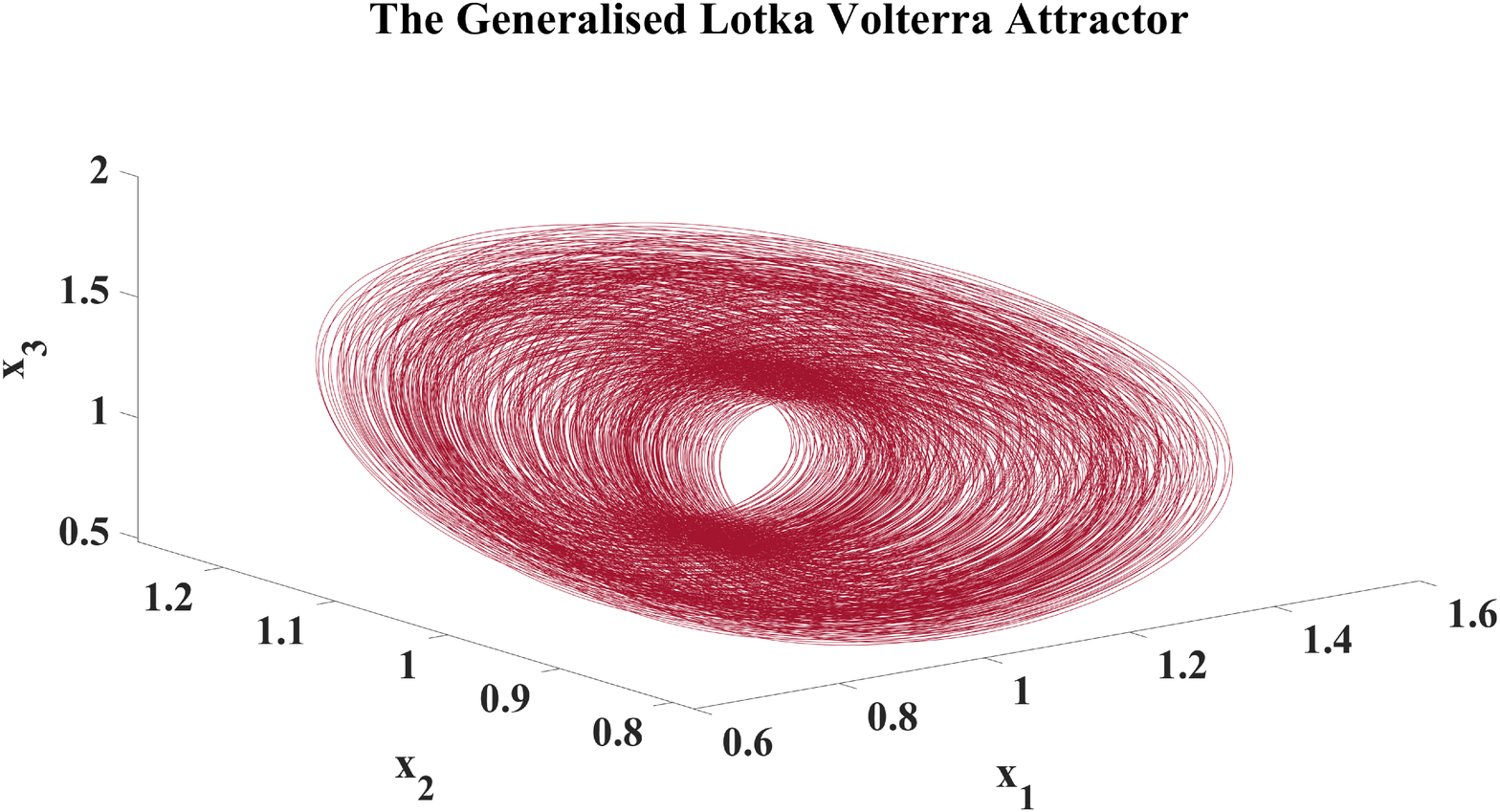}}  &
 \subfigure[]{\includegraphics[width=6.5cm,height=3cm]{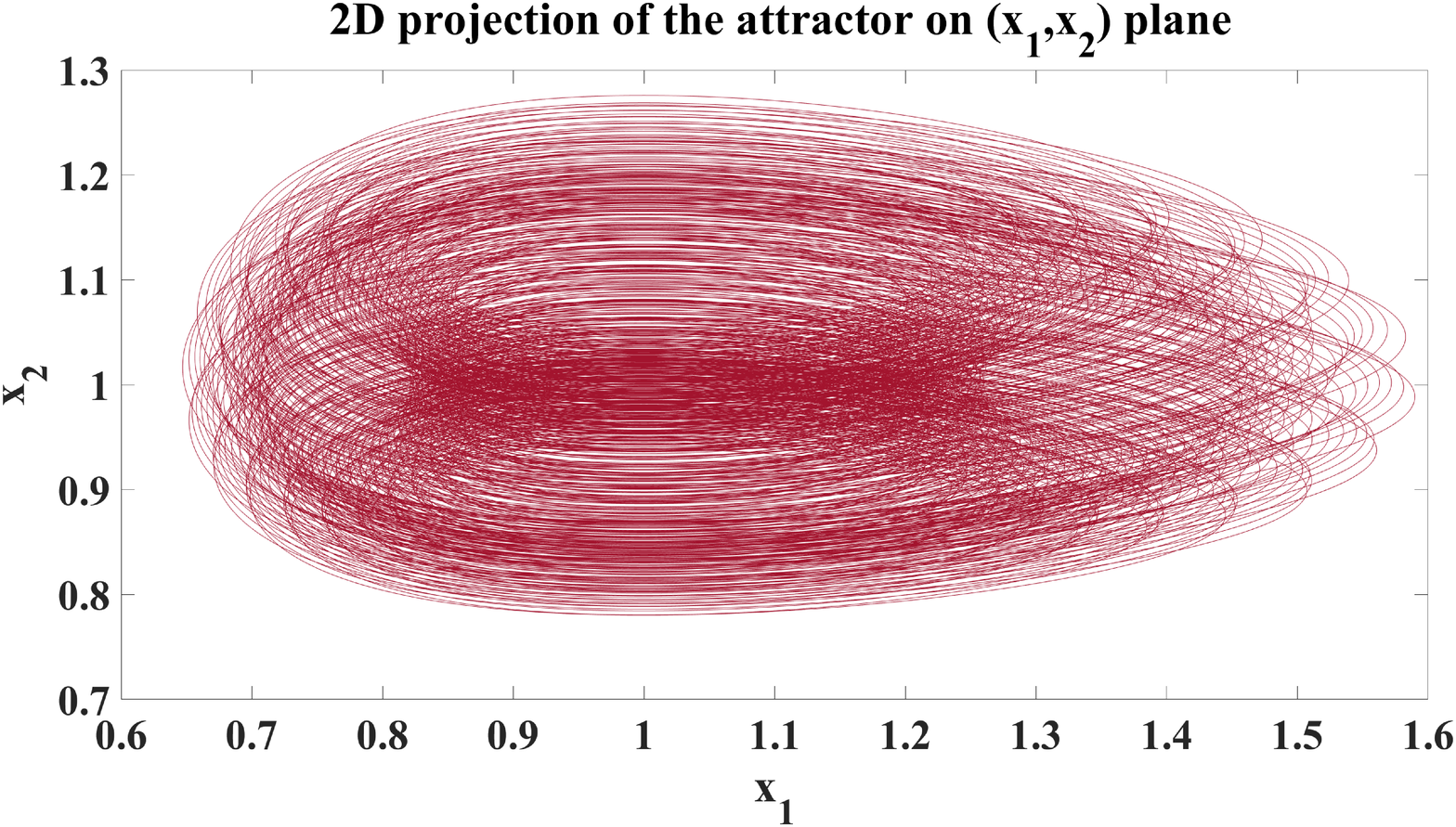}}    \\
 \subfigure[]{\includegraphics[width=6.5cm,height=3cm]{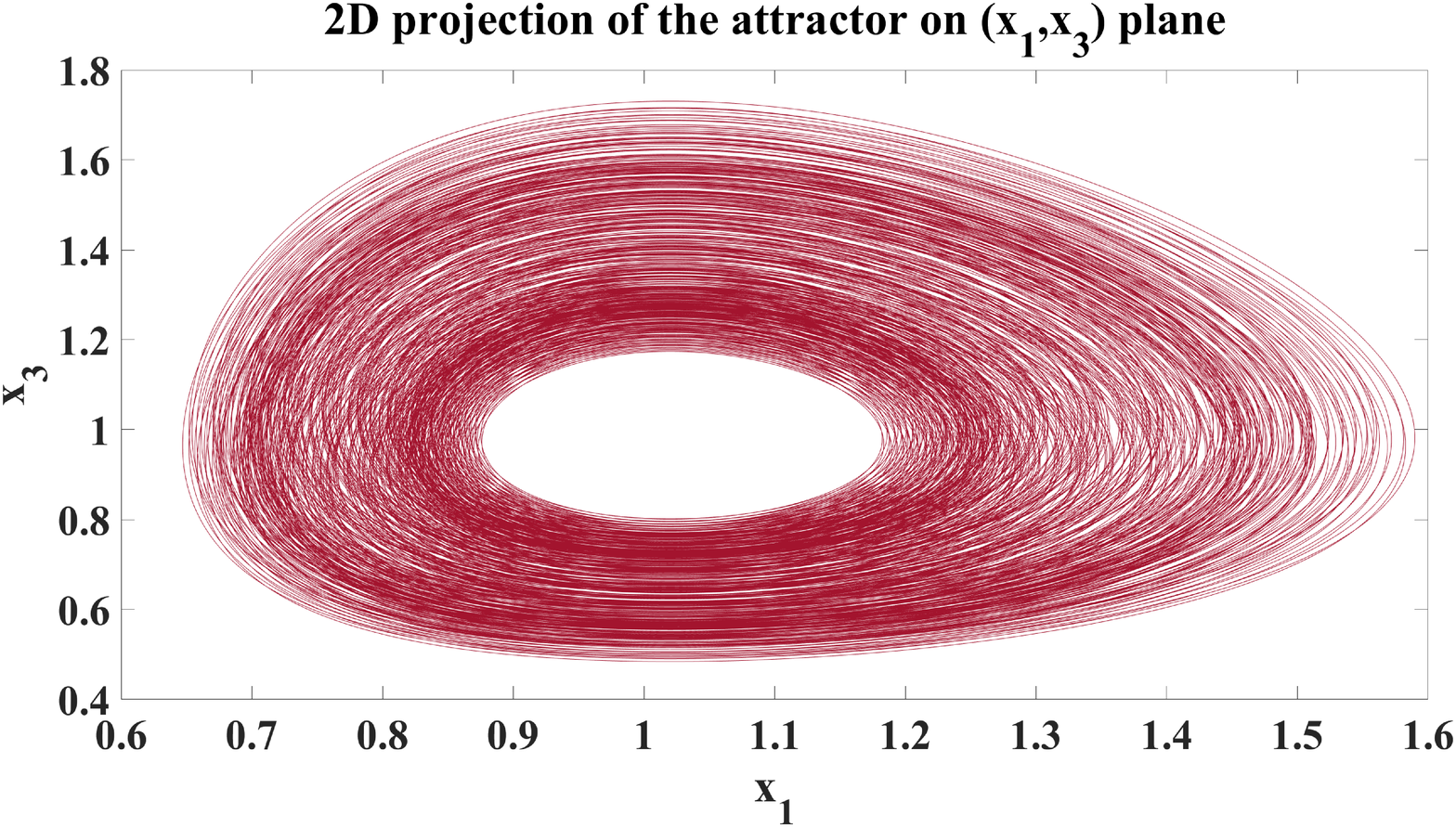}}    & 
  \subfigure[]{\includegraphics[width=6.5cm,height=3cm]{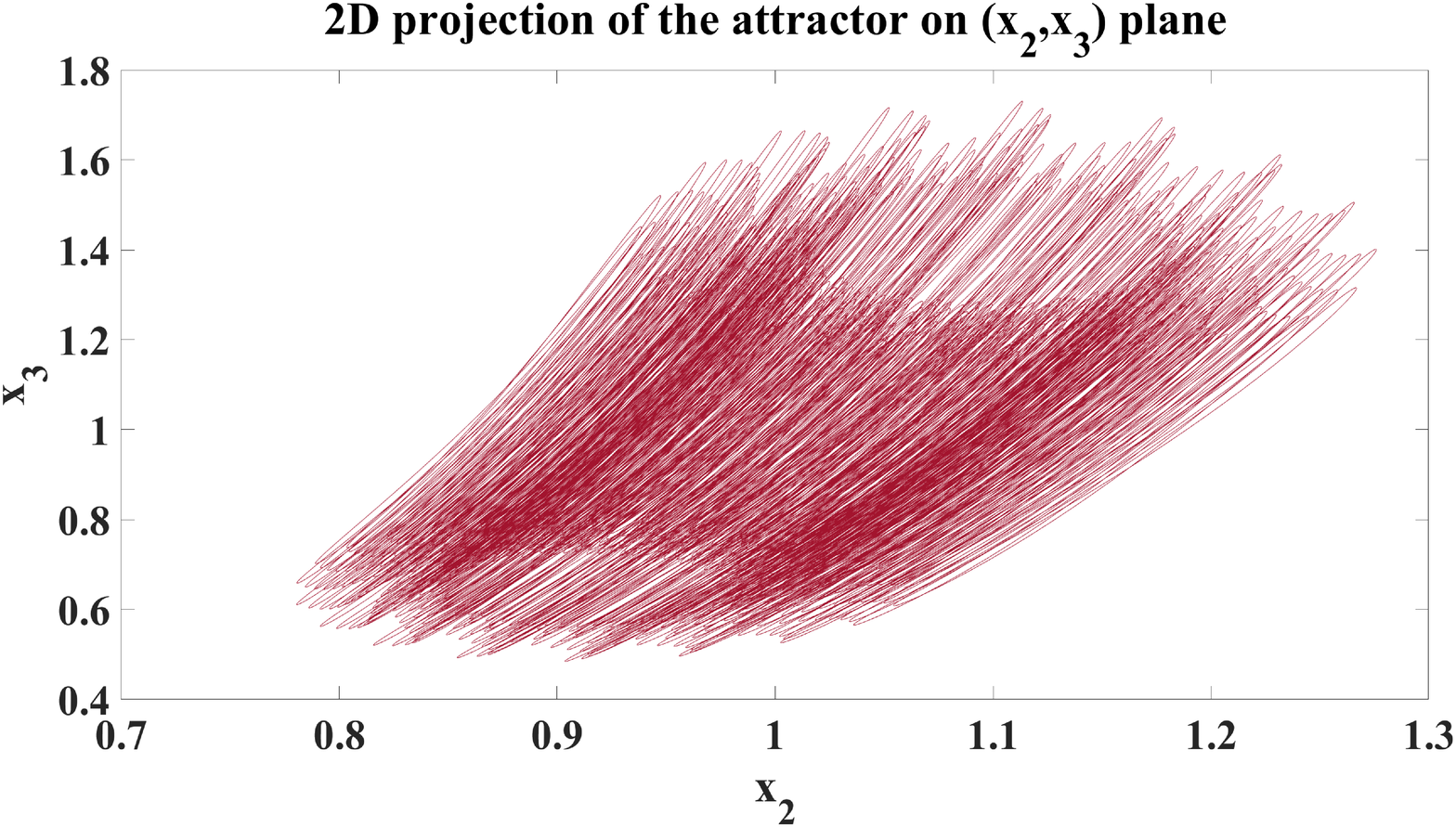}}  
\end{tabular}
\caption{(a): 3D attractor; (b), (c) and (d): 2D projections of the attractor on $({x}_{1}, {x}_{2}), ({x}_{1},{x}_{3})$ and $({x}_{2}, {x}_{3})$ planes respectively for $(p, q, r) = (2.0451, 2.129,  2)$.}
\label{fig:2} 
\end{figure}
 
\begin{figure}[ht]
\centering
\begin{tabular}{cc}
 \subfigure[]{\includegraphics[width=6.5cm,height=3cm]{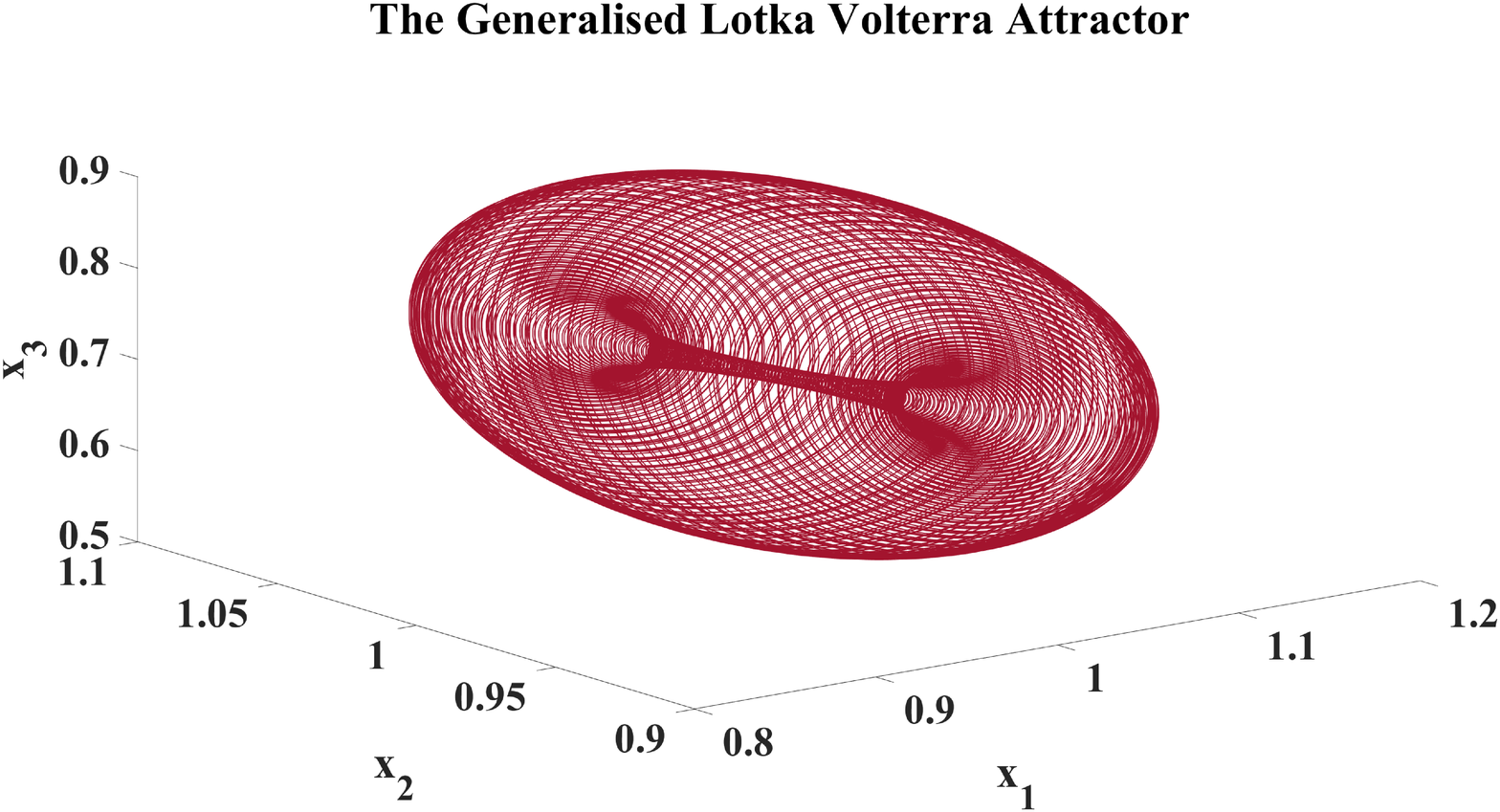}}  &
 \subfigure[]{\includegraphics[width=6.5cm,height=3cm]{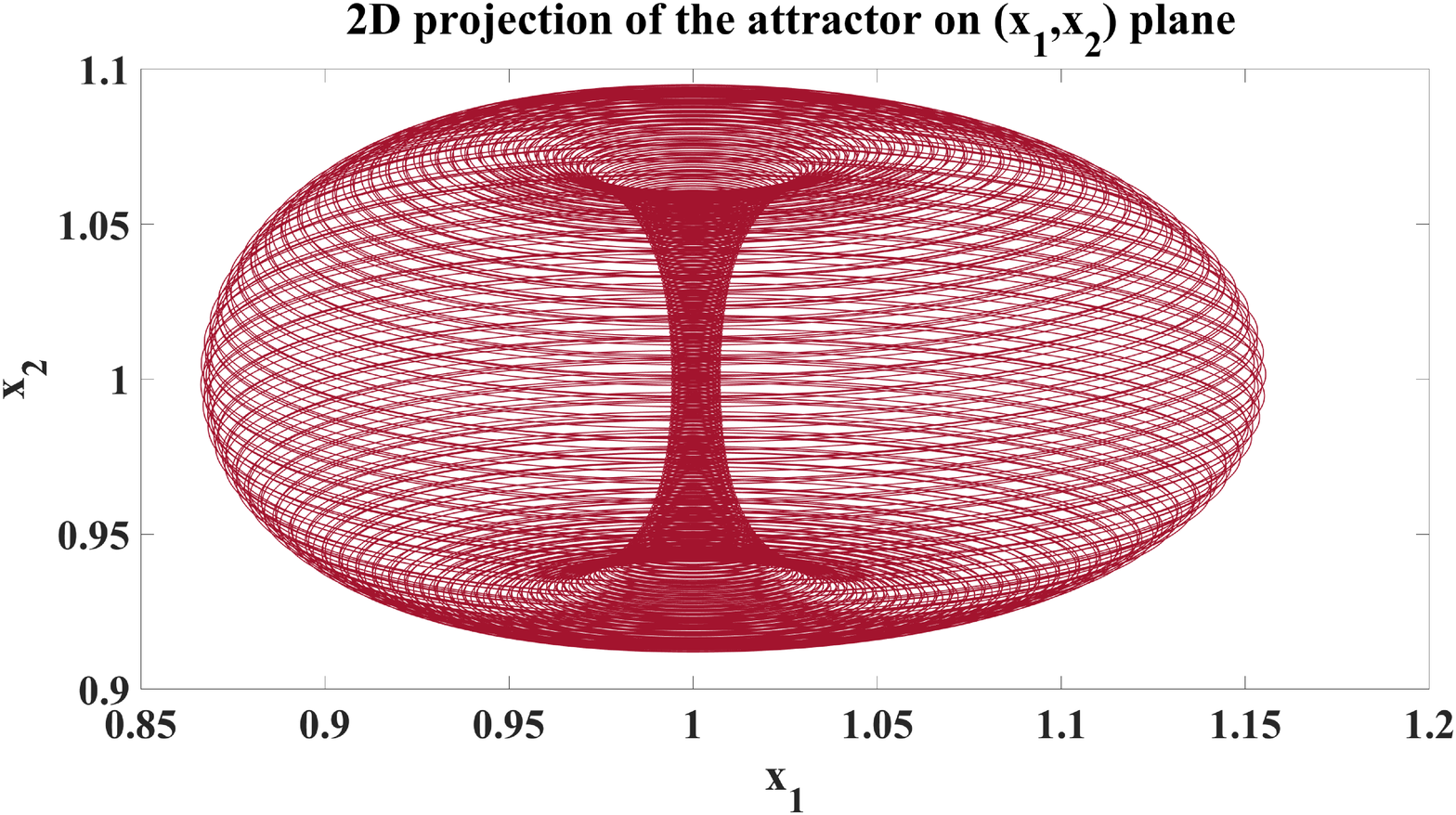}}   \\
  \subfigure[]{\includegraphics[width=6.5cm,height=3cm]{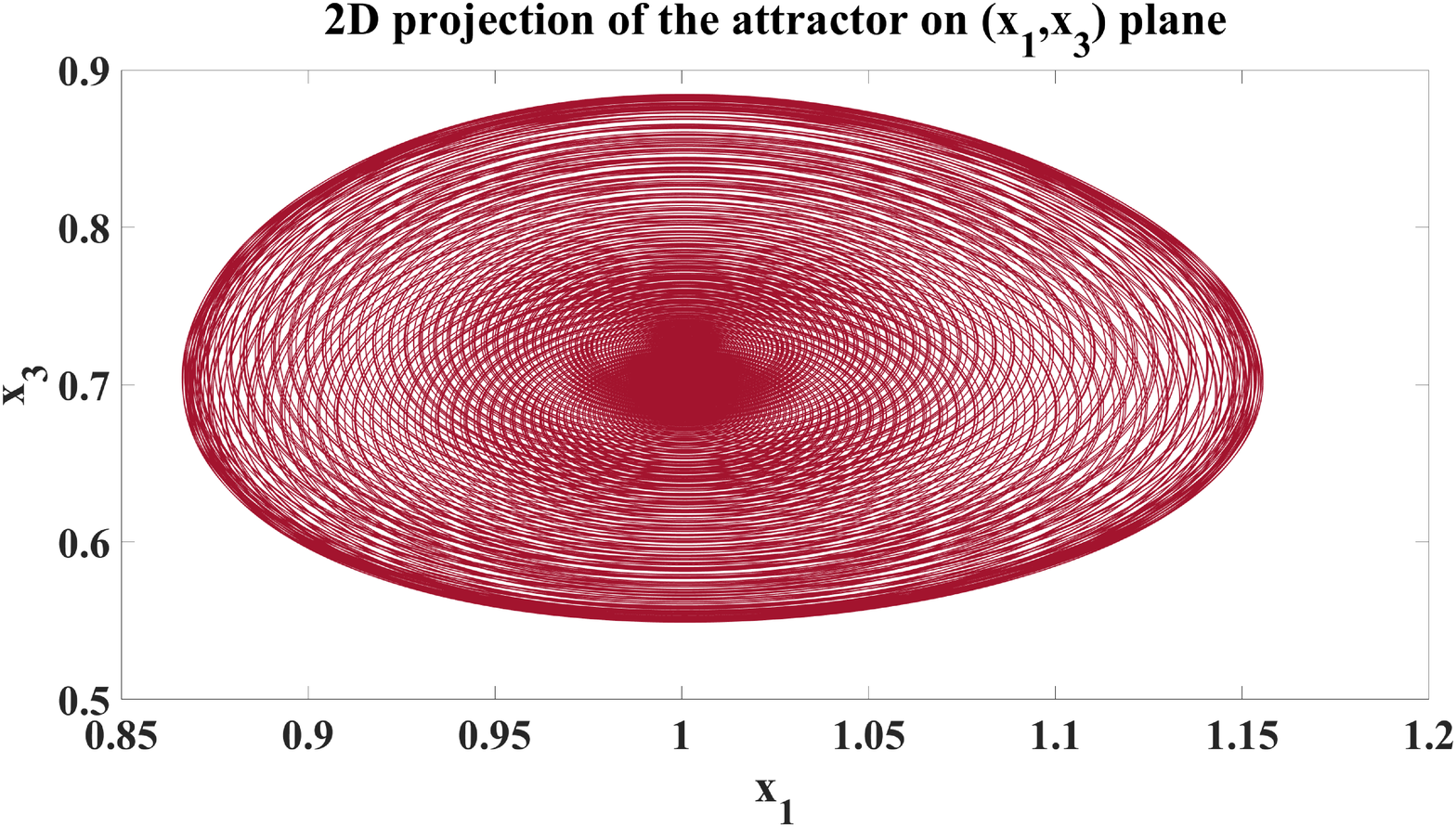}}  &
  \subfigure[]{\includegraphics[width=6.5cm,height=3cm]{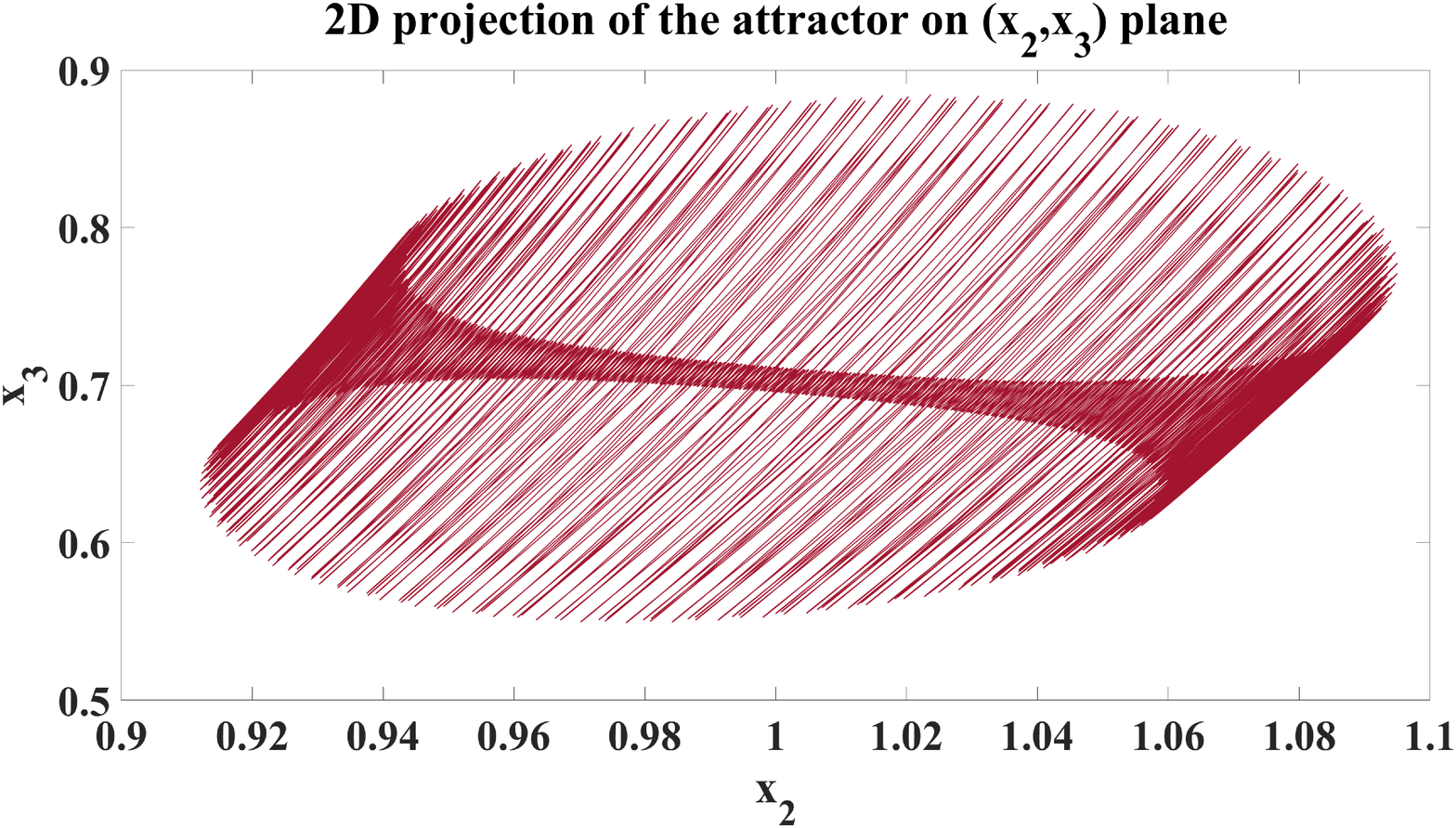}}   
\end{tabular}
\caption{(a): 3D  attractor ; (b), (c) and (d): 2D projections of the attractor on $({x}_{1}, {x}_{2}), ({x}_{1},{x}_{3})$ and $({x}_{2}, {x}_{3})$ planes respectively for  $(p,q,r) = (2.9851, 2.99, 2.1)$.}
\label{fig:21} 
\end{figure}

\begin{figure}[ht]
\centering
\begin{tabular}{cc}
 \subfigure[]{\includegraphics[width=6.5cm,height=3cm]{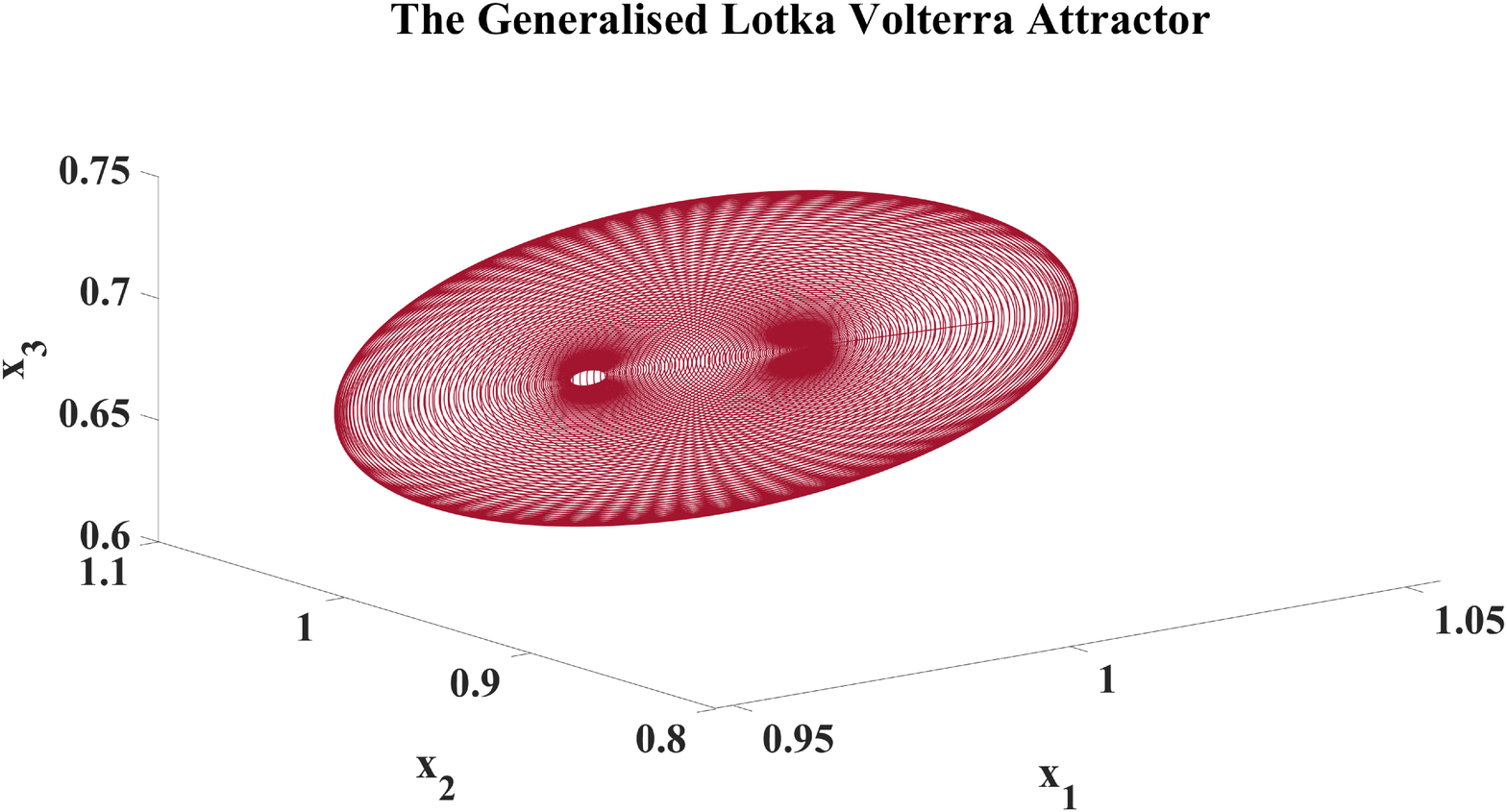}}  &
 \subfigure[]{\includegraphics[width=6.5cm,height=3cm]{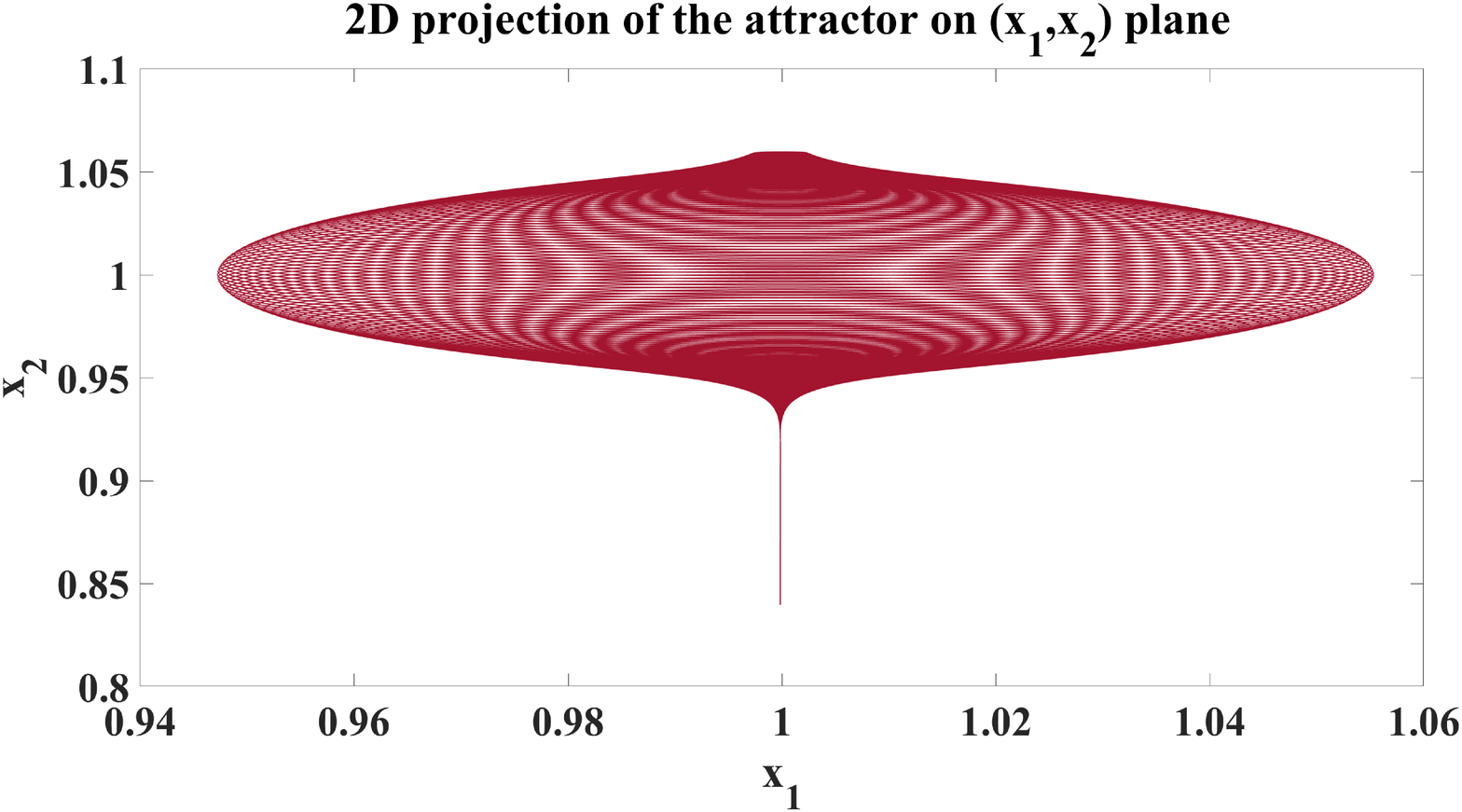}}   \\
  \subfigure[]{\includegraphics[width=6.5cm,height=3cm]{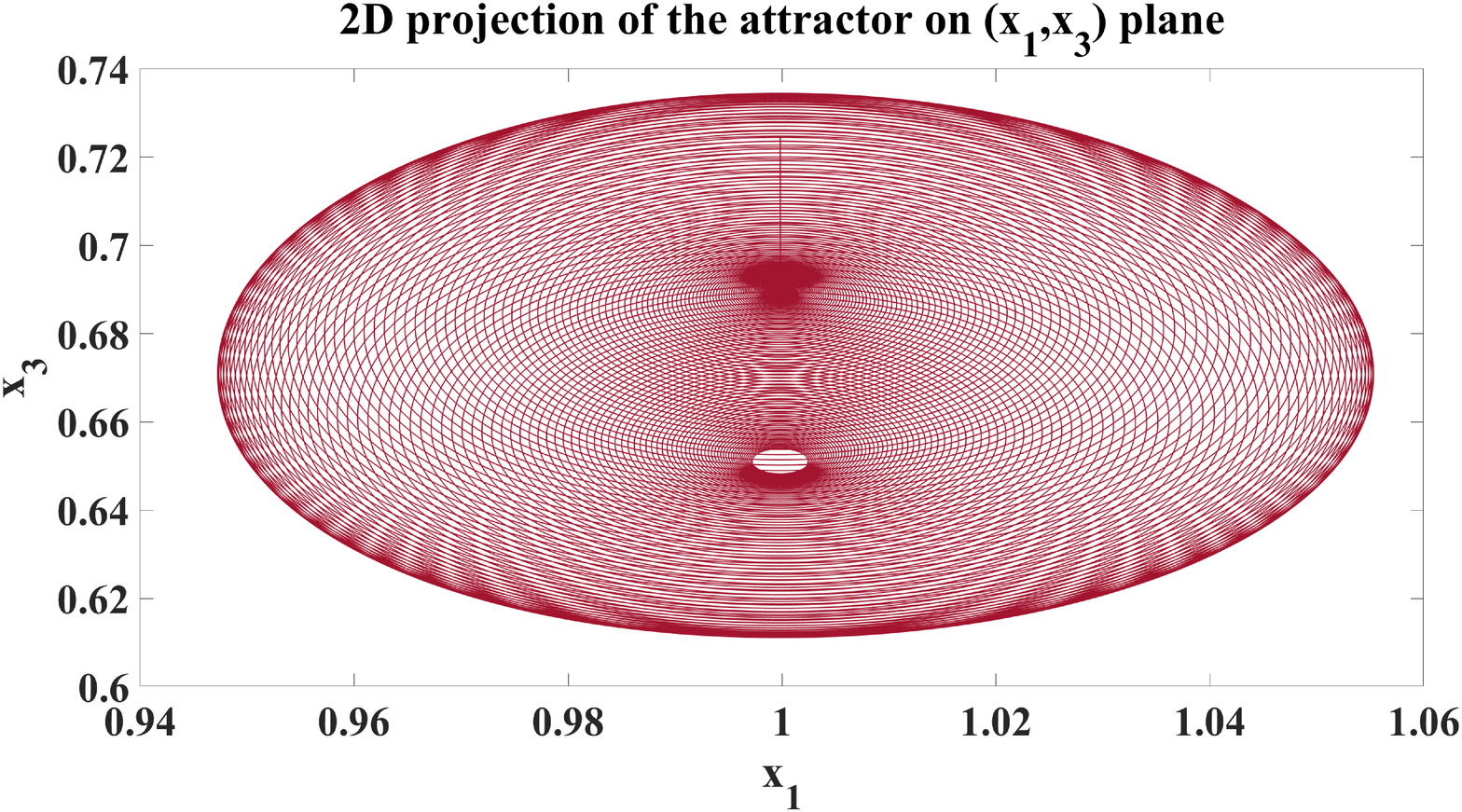}}  &
  \subfigure[]{\includegraphics[width=6.5cm,height=3cm]{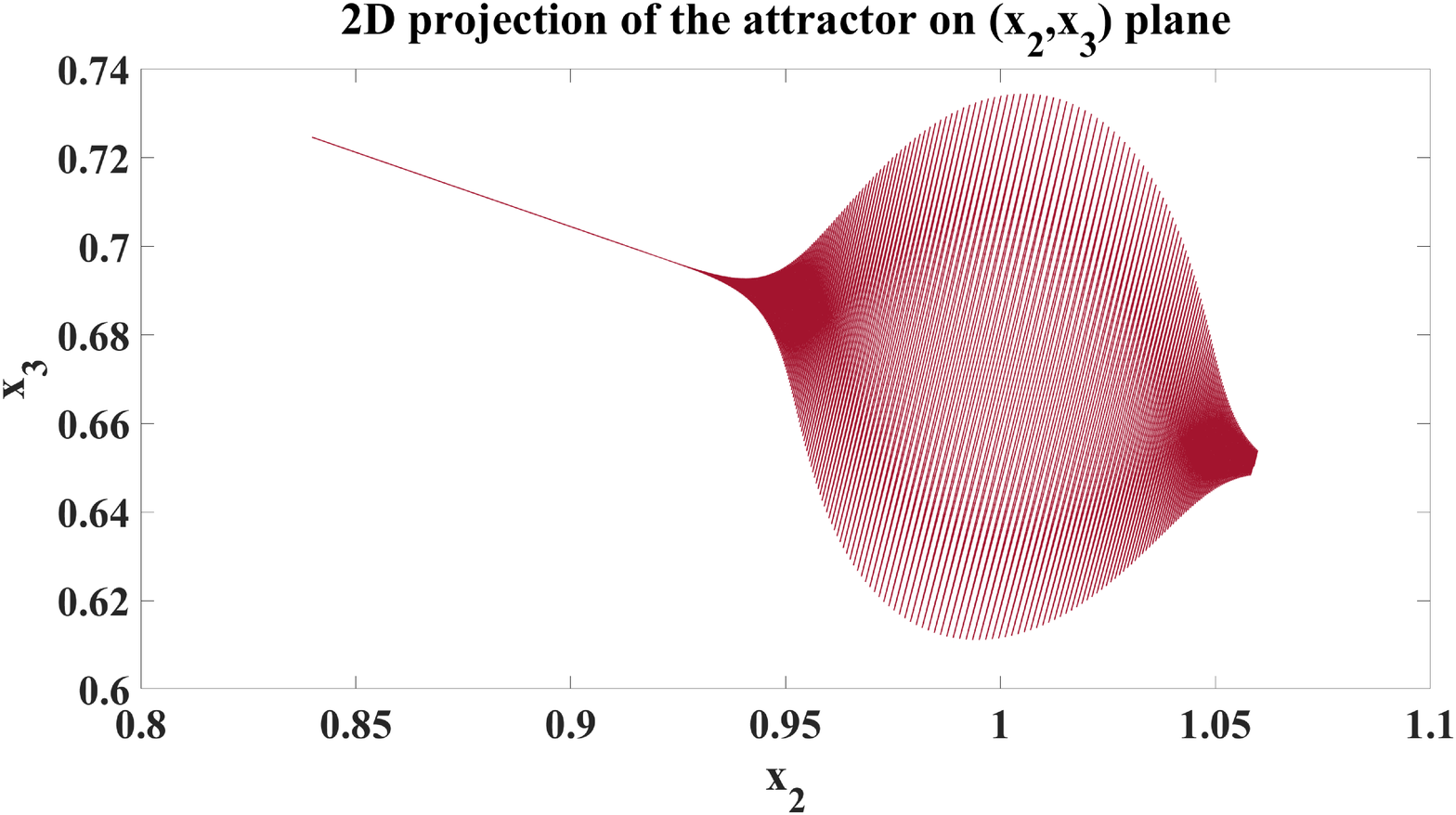}}   
\end{tabular}
\caption{(a): 3D attractor; (b), (c) and (d): 2D projections of the attractor on $({x}_{1}, {x}_{2}), ({x}_{1},{x}_{3})$ and $({x}_{2}, {x}_{3})$ planes respectively for  $(p,q,r) = (2.98098, 2.9799, 2)$.}
\label{fig:22} 
\end{figure}

Next subsections presents two other variants of GLV model with different functional responses.
\subsection{\label{sec2.2}Model with Cyrtoid type (HT II)  functional response}
To elucidate the role of functional response, we replace the linear interaction term  between prey and middle predator with Holling type II functional response ( $(\frac{{x}_{1}}{{x}_{1}+d}){x}_{2}$). The ecological meaning of the non-linear interaction of prey with middle predator is that the prey's contribution to the middle predator growth rate is $\frac{{x}_{1}{x}_{2}}{{x}_{1}+d}$. Using type II functional response, the dynamics of new GLV model are proposed as
\begin{equation}
\begin{split}
    \dot{{x}_{1}} & = {x}_{1}-\frac{{x}_{1}{x}_{2}}{{x}_{1}+d}+{r}{x}_{1}^{2}
       -{p}{{{x}_{1}}^{2}}{x}_{3},\\
    {\dot{{x}_{2}}} &= -{x}_{2} +(\frac{{x}_{2}{x}_{1}}{{x}_{1}+d}), \\ 
    {\dot{{x}_{3}}}& = -{q}{x}_{3}+{p}{x}_{3}{x}_{1}^{2}. 
    \end{split}
\end{equation}
For simulation, we take the parametric values and initial condition  as $p=2.514, q=2.9089, r=2.1990507, d=.00198$ and $ {x}_{1}(0)= 1.78  ,~ {x}_{2}(0)=   0.5020,~ {x}_{3}(0)=  1.01$ respectively. Figure \ref{fig:4} displays the three-dimensional phase portrait of GLV system with HT II functional response which is a `stable focus'. Thus, a change in functional response in GLV system can lead to stable dynamics for a suitable set of parameter values and initial conditions.

\begin{figure}[ht]
\centering
\begin{tabular}{c}
 \subfigure[]{\includegraphics[width= 8cm, height=4cm]{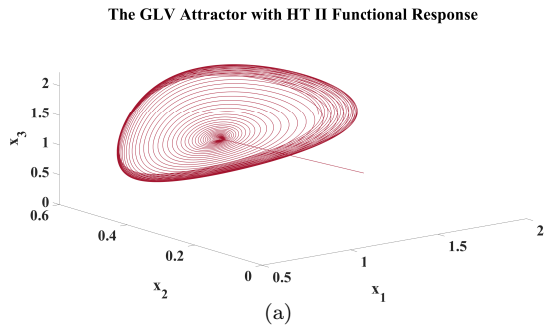}}   
\end{tabular}
\caption{(a) Three-dimensional attractor of GLV equations with  HT II functional response.}
\label{fig:4}
\end{figure}

\subsection{\label{sec2.3} Model with Sigmoid type (HT III) functional response}
Next, we change interaction term between prey and top predator population with Holling type III functional response. The inclusion of  Holling type III functional response increases the search activity of  top-predator for prey. With assumption of increasing prey density, the dynamics of new GLV model with Holling type III functional response are proposed as
\begin{equation}
\begin{split}
 \dot{{x}_{1}}& = {x}_{1}-{x}_{1}{x}_{2}+{r}{x}_{1}^{2}-{p}\frac{{x}_{1}^{2}{x}_{3}}{{x}_{1}^2+d},\\
    \dot{{x}_{2}} &= -{x}_{2} +{x}_{1}{x}_{2},\\ 
    \dot{{x}_{3}}& = -{q}{x}_{3}+{p}\frac{{x}_{1}^{2}{x}_{3}}{{x}_{1}^2+d}. 
\end{split}
\end{equation}
For parameter values $ p= 7.34,~ q= 2.0,~ r=0.507,~ d=3.198$, the system (3) has  `limit cycle'-like attractor which means model (3) has stable dynamics.
\begin{figure}[ht]
\begin{center}
 \includegraphics[width= \textwidth,height=5cm]{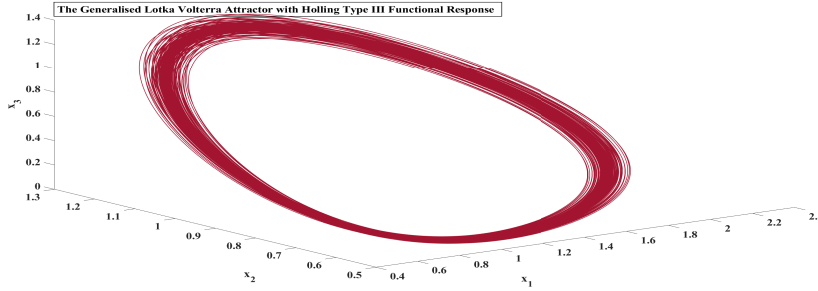}
 \caption{Three-dimensional attractor of generalised Lotka Volterra equations with HT III functional response.}
 \end{center}
 \label{fig:5} 
\end{figure}
\\
For each variation, we have found  different parameters values for which models (1), (2), and (3) show completely different dynamics. We infer that the GLV model's unstable dynamics with linear function response can be turned into stable dynamics when linear functional response is altered by HT II or HT III. However, this may not be very effective for arbitrarily given scenario. A case in point-  predators usually do not follow predation rate as  HT II or HT III functional response when the prey population is in abundance. To overcome this problem, we control chaos in the model (1) through synchronization and achieve complete replacement synchronization in two coupled GLV models following linear functional response. Since GLV system has been shown to have a chaotic attractor for various set of parameters. Out of these sets, we pick the set of  parameters $ p= 2.9851, q= 3, r=2$ as in \cite{elsadany2018dynamical} for further study.
\section{\label{sec3} Mathematical Properties}
In this section, we discuss the dynamical and analytical properties of system $(1)$  including positive Lyapunov exponent, equation of slow manifold, in-variance, dissipation, stability of feasible equilibrium points, and  control of instability of unstable equilibrium points. 
\subsection{\label{sec3.1} Lyapunov exponents} 
For three-dimensional system, the local behaviour of the dynamics varies  along three orthogonal directions in state space. In a given chaotic system , nearby initial conditions may be moving apart along one axis, and moving together along another. The Lyapunov exponent describes the average rate of separation between two nearby trajectories with different initial conditions subject to a flow \cite{solari1996nonlinear}. Where a positive Lyapunov exponent confirms chaos in the system. For simulation, we take the parameter values of the system $(1)$ as $ p = 2.0451, q = 2.129, r = 2 $. The dynamics of Lyapunov exponents are shown in figure \ref{fig:6}.
\begin{figure}[ht]
\begin{center}
\includegraphics[width=\textwidth, height= 7cm]{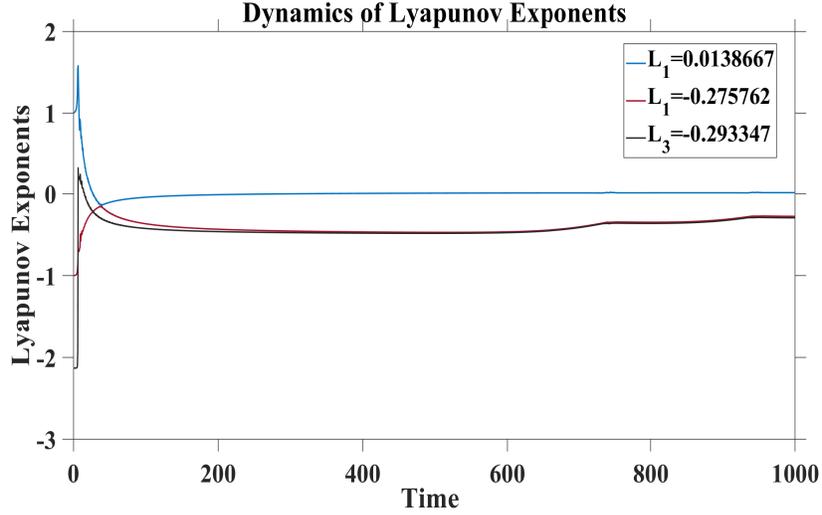}
\caption{Dynamics of Lyapunov spectrum of system $(1)$.}
\label{fig:6} 
\end{center}
\end{figure}
The Lyapunov exponents of model (1) are as follows 
$$
{L}_{1} = 0.0138667>0, {L}_{2} = -0.275762 <0, {L}_{3}=-0.293347 < 0.  
$$
where $L_{1}$ is the indicator of chaos in the system (1).
 \subsection{\label{sec:leve22}Equation of slow manifold}
The infusion of geometric and topological techniques in chaos theory motivates mathematicians to study the underlying geometric structures. In this line, expression of slow manifold permits to restore a part of the deterministic property of the system that was lost because of SDIC. To find an equation of slow manifold, we consider the system (1) as slow-fast autonomous dynamical system (S-FADS). In S-FADS, variables are separated into two groups:, one is group of fast variable and other is of slow variables where slow variables are used to determine the behaviour of whole system. To get the equation, we consider that the slow manifold is locally defined by a plane orthogonal to tangent system's left fast eigenvector.  Under the set of parameter values $ p= 2.9851, q= 3, r=2$,  the equations of GLV model (1) can be given  as 
\begin{equation}
    \begin{array}{ll}
    \dot{{x}_{1}}= {x}_{1}(1-{x}_{2}+2{x}_{1}-2.9851{x}_{3}{x}_{1}),\\
    \dot{{x}_{2}}= {x}_{2}(-1 +{x}_{1}), \\ 
    \dot{{x}_{3}}= {x}_{3}(-3+2.9851{x}_{1}^{2}). 
\end{array} 
\end{equation}  
 The Jacobian matrix $J$ at point $ \textbf{x}= ({x}_{1},{x}_{2},{x}_{3})^T$ is obtained as
$$J=\begin{bmatrix}
1-{x}_{2}+4{x}_{1}-5.9702{x}_{1}{x}_{3} & -{x}_{1}  & -2.9851{x}_{1}^{2}\\
                               {x}_{2}        &-1+{x}_{1} &  0\\
          5.9702{x}_{1}{x}_{3}&    0       &-3+2.9851{x}_{1}^{2}\\
    \end{bmatrix}$$
Let ${\lambda}_{1}({x}_{1},{x}_{2},{x}_{3})$ be a real, negative and dominant Eigen value ( i.e, fast Eigen value) for Jacobian matrix in a large part of attractor's phase space domain.
\\
 Furthermore, we assume that ${\lambda}_{2}({x}_{1},{x}_{2},{x}_{3})$ and ${\lambda}_{3}({x}_{1},{x}_{2},{x}_{3})$ be two slow Eigen values. Then Eigen vector ${Z}_{{\lambda}_{1}}^{T}$ corresponding to fast Eigen value ${\lambda}_{1}$ of $J^{T}({x}_{1},{x}_{2},{x}_{3})$ is given by
\begin{equation}
|J - {{\lambda}_{1}}I|~{Z}_{{\lambda}_{1}} = 0.
\end{equation}
where $I$ is $3 \times 3$ identity matrix. Equation (6) gives,
$${Z}_{{\lambda}_{1}}^{T}=\begin{bmatrix}
(-1+{x}_{1}-\lambda_{1})(-3+2.9851{x}_{1}^{2}-\lambda_{1})\\
                          x_{1}(-3+2.9851{x}_{1})^{2})  \\
          2.9851{x}_{1}^{2}(-1+{x}_{1}-\lambda_{1})\\
    \end{bmatrix}.$$
On the attractive parts of phase space (where $J(\textbf{x})$ has a fast Eigen value $\lambda_{1}$), the equation of the slow manifold is given by  
\begin{equation}
\dot{\textbf{x}}(t).{Z}_{{\lambda}_{1}}^{T} = 0.
\end{equation}
We use the equation $(7)$ to define the equation of slow manifold. With the substitution of $\dot{\textbf{x}}(t)$ and ${Z}_{{\lambda}_{1}}^{T}$ in the equation $(7)$, we write the equation of slow manifold as
\begin{equation} 
\begin{split}
 & {\lambda}_{1}^{2}({x}_{1}-{x}_{1}{x}_{2}+2{x}_{1}^{2}-2.9851{x}_{1}^{2}{x}_{3}) +  {\lambda}_{1}(-4{x}_{1}+7{x}_{1}^{2} 
 -4.9851{x}_{1}^{3}-3{x}_{1}{x}_{2}+2.9851{x}_{1}^{3}{x}_{2} -
 \\ &5.97020{x}_{1}^{4}-2.985100{x}_{1}^{2}{x}_{3}+2.9851{x}_{1}^{3}{x}_{3})+(3{x}_{1}+{x}_{1}^{2}-8.985100{x}_{1}^{3}-2.9851{x}_{1}^{4}+5.970200{x}_{1}^{5} \\
 & +8.95530{x}_{1}{x}_{3}-17.910600{x}_{1}^{2}{x}_{3}-0.044478{x}_{1}^{3}{x}_{3}+17.821644{x}_{1}^{4}{x}_{3}-8.91082{x}_{1}^{3}{x}_{3})=0.
\end{split}
\end{equation}
 where ${\lambda}_{1}$ is fast Eigen value of $J(\textbf{x})$. Because $\lambda_{1} (x_{1}, x_{2}, x_{3})$ is uncertain Eigen value, it is not easy to use this implicit equation to draw a slow manifold representation in the three dimensional phase space.

\subsection{\label{sec3.3}Invariance property}
\begin{theorem}
Let the system in vector notation is given as 
 \begin{equation}
  \dot{\textbf{x}}(t) = H(\textbf{x(t)}) = \begin{bmatrix} 
             {H}_{1}({x}_{1},{x}_{2},{x}_{3}) \\
              {H}_{2}({x}_{1},{x}_{2},{x}_{3}) \\
              {H}_{3}({x}_{1},{x}_{2},{x}_{3}) 
    \end{bmatrix} 
    \end{equation}
\[  
\begin{split}
{H}_{1}({x}_{1},{x}_{2},{x}_{3}) &= {x}_{1}(1-{{x}_{2}}+{r}{x}_{1}-p {x}_{3}{x}_{1}),\\
   {H}_{2}({x}_{1},{x}_{2},{x}_{3}) &= {x}_{2}(-1 +{x}_{1}), \\ 
    {H}_{3}({x}_{1},{x}_{2},{x}_{3}) &= {x}_{3}(-{q}+{p}{{x}_{1}}^{2}). 
\end{split}\]

  where ${H}_{1}, {H}_{2}$ and ${H}_{3}$ are continuously differentiable. Assume that $\textbf{\textit{H}}$ is locally Lipschitz and generates a flow $\phi_{t}(\textbf{x})$. Let 
  $$L: D \subset{R}^{3} \rightarrow {R}^{3}$$
 be a continuously differentiable function on a domain $D \subset {R}^{3} $ such that $ \dot{L}(\textbf{x})\leq 0$ in $D$, then the largest invariant set $\Sigma \subset D $ is the set; where ${\nabla{L}}.{H(\textbf{x})} = 0 ~~ \forall \textbf{x} \in \Sigma$.
\end{theorem}
 \begin{proof}  Consider 
  $$L: D \subset{R}^{3} \rightarrow {R}^{3}$$
   be a continuously differentiable function on a domain $ D \subset {R}^{3} $ and defined as
   \begin{equation}
    L({x}_{1},{x}_{2},{x}_{3}) = \frac{{x}_{1}^2+{x}_{2}^2+{x}_{3}^2}{2}.
   \end{equation}
 Equation (9) gives,
 \begin{equation}
 \dot{L}{({x}_{1},{x}_{2},{x}_{3})} = {x}_{1}\dot{{x}}_{1}+{x}_{2}\dot{{x}}_{2}+{x}_{3}\dot{{x}}_{3}.
 \end{equation}
 The set  $D \subset{R}^{3}$ is said to be an invariant set under the flow $\phi_{t}$ if
 for any point $\textbf{x} \in D$
  $$ ~\phi_{t}(\textbf{x}) \in D ~\forall~ t \in {R}.$$
 Let $\Sigma$ be a smooth closed surface without boundary in $D \subset{R}^{3}$ and suppose that $\textbf{\textit{n}}$ is a normal vector to the surface ${\Sigma}$ at $({x}_{1},{x}_{2},{x}_{3})$. If we have 
 \begin{equation}
   \textbf{\textit{n}}.<\dot{x}_{1},\dot{x}_{2},\dot{x}_{3}>=0 ~~ \forall~~({x}_{1},{x}_{2},{x}_{3}) \in \Sigma.
 \end{equation}
Let us consider $\Sigma$ be the ${x}_{1}{x}_{2}$ plane i.e. ${x}_{3} = 0$. Note that the vector $\textbf(0, 0, 1)$ is always normal to $\Sigma$ and at the point $({x}_{1},{x}_{2}, 0) \in \Sigma$. So we have,
$$({\dot{x}_{1}},{\dot{x}_{2}},{\dot{x}_{3}})= ({{x}_{1}}(1-{x}_{2}+r{x}_{1}-p {x}_{3} {x}_{1}),{{x}_{2}}(-1 +{x}_{1}),{0}).$$
Thus,
$$\langle{({0},{0},{1}).({{x}_{1}}(1-{x}_{2}+r{{x}_{1}}-p{{x}_{3}}{{x}_{1}}),{{x}_{2}}(-1 +{x}_{1}),{0})} \rangle = {0}.$$ 
Similar arguments can be verified for ${{x}_{1}}$ and ${{x}_{2}}$ planes which directs that each coordinate plane is an invariant subset. It implies that for any given positive initial condition, ${{x}_{1}(t)}, {{x}_{2}(t)}$ and ${{x}_{3}(t)}$ are positive for all ${t}$ that is any trajectory starting in ${R}^{3}_{+}$ can not cross the co-ordinate planes and it shows that ${R}^{3}_{+}$ is an invariant set for the system. 
\end{proof} 
\subsection{\label{sec3.4}Dissipation}
\begin{theorem}
Consider the autonomous vector field 
$${\dot{\textbf{x}}(t)= H(\textbf{x})}~ for ~\textbf{x}\in {R}^{3},$$
 and assume that it generates a flow $\phi_{t}(\textbf{x})$. Let ${D}_{0}$ is a domain in ${R}^{3}$ which is supposed to have a volume ${V}_{0}$, and ${\phi}_{t}({D}_{0})$ is its evolution under the flow. If $V(t)$ is the volume of ${D}_{t}$, then the time rate of change of volume is given as
$$|\frac{d{V}}{d{t}}|_{t=0} = \int_{{D}_{0}}{\nabla.{H}} d{\textbf{x}}. $$ 
The system (1) is dissipative if its time-$t$ map decreases volume for all ${t}>{0}$.
\end{theorem}
\begin{proof}  Dissipation in any dynamical system manifests itself as contraction of the phase volume on average. To check this, we express the volume $V(t)$ in the following form using the definition of the Jacobian of  transformation as
\begin{equation}
V(t)= \int_{D_{0}}|\frac{d\phi_t{(\textbf{x})}}{d{\textbf{x}}}| d{\textbf{x}}. 
\end{equation}
Expanding $\phi_t{(\textbf{x})}$ in the neighbourhood of ${t = 0}$. Since the vector field $H(\textbf{x})$ is smooth enough to have a tangent plane in each point on ${R}^{3}$ so we can expand  ${\phi}_t{(\textbf{x})}$ by Taylor series expansion. Hence we get,
\begin{equation}
\phi_{t}{(\textbf{x})} = \textbf{x} + \dot{\textbf{x}}{t} +O(t^2)~ for~{t}~\rightarrow {0} 
\end{equation}
Since 
\begin{equation}
{\dot{\textbf{x}}(t)= H(\textbf{x})},
 \end{equation}
 The equation (15) gives,
\begin{equation}
{\phi}_{t}{(\textbf{x})} = \textbf{x} + H(\textbf{x}) t +O(t^2) ~for~{t}~\rightarrow {0}.
\end{equation}
 It follows that
 $$ \frac{\partial{\phi}}{\partial{\textbf{x}}} = I +\frac{\partial{H}}{\partial{\textbf{x}}}t + O(t^2),$$
 \begin{equation}
|\frac{\partial{\phi}}{\partial{\textbf{x}}}| = |I +\frac{\partial{H}}{\partial{\textbf{x}}}t| +O(t^2).
\end{equation}
 Here $I$ is ${3 \times 3 }$ identity matrix so $\det{I}$ will be equal to ${1}$. By expanding the expression $(17)$ by using expansion of determinant, we get the following 
\begin{equation}
 |\frac{\partial{\phi}}{\partial{\textbf{x}}}| = 1+trace(\frac{\partial{H}}{\partial{\textbf{x}}})t +O(t^{2}).
 \end{equation} 
 Note that
 \begin{equation} 
  trace(\frac{\partial{H}}{\partial{\textbf{x}}}) = {\nabla.{H}},
  \end{equation} 
 therefore, we have
 \begin{equation}
   V(t) = {V}_{0} + {\int}_{{D}_{0}} ((\nabla.{H})t + O(t^{2}))d{\textbf{x}}.
   \end{equation}
  It gives
  \begin{equation}
   |\frac{d{V}}{d{t}}|_{t=0} = \int_{D_{0}} {\nabla .{H}} d{\textbf{x}}, 
\end{equation}  
i.e. if the volume shrinks then divergence of vector field will be strictly negative \cite{solari1996nonlinear}. 
\\
Now considering the equations of model $(1)$ in vector notation and computing its Jacobian
\begin{equation}
J({x}_{1},{x}_{2},{x}_{3})=\frac{\partial{H}}{\partial{\textbf{x}}}.
\end{equation}
The Jacobian  $J({x}_{1},{x}_{2},{x}_{3})$ of the model is given by
\begin{equation}
\begin{bmatrix}
1-{x}_{2}+2{r}{{x}_{1}}-2{p}{{x}_{1}}{{x}_{3}}&{-{x}_{1}}&{-p}{{x}_{1}}^{2}\\
{{x}_{2}}&{-1+{x}_{1}}&{0}\\
2{p}{{x}_{1}}{{x}_{3}}&{0}&-q+{p}{x}_{1}^{2}\\
\end{bmatrix},
\end{equation}
we take the parameter values as
\begin{equation}
 p= 2.9851, \quad  q= 3, \quad r=2.
 \end{equation}
The above argument shows that the G.L.V dynamical system will be dissipative if the generalized divergence should be less than zero, i.e.
\begin{equation}
  {\sum}_{i}\frac{\partial {H}_{i}}{\partial {x}_{i}} < 0.
   \end{equation}
 where Einstein summation has been used. The divergence of vector field ${H}$ on ${R}^{3}$ is as follows
 \begin{equation}
 \begin{array}{ll}
 {\nabla}.{H} = \frac{\partial {H}_{1}}{\partial {x}_{1}}+\frac{\partial {H}_{2}}{\partial {x}_{2}} + \frac{\partial {H}_{3}}{\partial {x}_{3}},\\
 {\nabla}.{H} = -{q} - {x}_{2}+(2 {r}+{1}+{p}{x}_{1}-{2}{p}{x}_{3}){x}_{1}.
 \end{array}
 \end{equation}
Hence system $(1)$ will be dissipative if the following condition is satisfied, 
\begin{equation}
(2{r}+{1}+{p}{x}_{1}-2{p}{x}_{3}){x}_{1}<{q}+{x}_{2}.
\end{equation} 
\end{proof}

\subsection{\label{sec:leve25}Existence and uniqueness of solution}
Since we are dealing with a population dynamics model, hence, the existence of at least one solution is must. However, the uniqueness of existed solution will give more appropriate results. Here we mention two theorems for which solutions of the system $(1)$ uniquely exist for all ${t}>{0}$ (complete detail of the proof can be seen in \cite{elsadany2018dynamical}).
\begin{theorem}
 If the functions ${f}_{1},~{f}_{2}$ and ${f}_{3}$ satisfy assumptions $(1)$ and $(2)$, mentioned in section \label{sec:level1}, then continuity of functions ${f}_{i}$ for $i \in \{1,2,3\}$ assures that  atleast one solution exists for the dynamics of system $(1)$ in region ${D} \times{I}$ where ${I} = (0,T]$ and the spatial boundary of  region $D \subset {R}^{3}$ is  defined as
$$D = \{\textbf{x}=({x}_{1},{x}_{2},{x}_{3}):~max{|{x}_{i}|} \leq {M},~for ~{i}\in\{ 1, 2,3\}\}$$
 where ${M}>{0}$. 
 \end{theorem}
\begin{theorem}
Let  $D$ be a closed subspace of complete normed linear space ${R}^{3}$. Consider $H : D \subset {R}^{3} \rightarrow D $ is Lipschitz continuous so that  there exist ${0}<{K}<{1}$ such that  
 $$||H(\chi)- H(\psi)|| < K||{\chi}- {\psi}||$$
 with $$ K = {T}.max (1+2{M}+2{r}{M}+4{p}{M}^{2}, 1+2{M}, q+2{p}{M}^{2})$$
For ${0}<{K}<{1}$, $\textbf{H}(t)$ will be a contraction map. With the help of Banach fixed point theorem, it can be ensured that ${0}<{K}<{1}$ is sufficient condition for uniqueness of  solution of the system $(1)$. 
\end{theorem}

\subsection{\label{sec3.6} Stability of feasible equilibrium points}
The equilibrium points of system $(1)$ are solutions of following algebraic equations
  \begin{equation}
\begin{split}
   {x}{x}_{1}(1-{x}_{2}+{r}{x}_{1}-{p}{x}_{3}{x}_{1}) = 0,\\
    {x}_{2}(-1 +{x}_{1}) = 0,\\ 
     {x}_{3}(-q+{p}{x}_{1}^{2}) = 0.
     \end{split}
\end{equation}
We obtain five equilibrium points by solving the system (27), 
\begin{equation} 
\begin{array}{ll}
   {X}_{0}^{*} = \begin{bmatrix} 
              0 \\
              0 \\
              0 
    \end{bmatrix}, \quad {X}_{1}^{*} = \begin{bmatrix} 
              1 \\ 
              1+r \\
              0 
    \end{bmatrix}, {X}_{2}^{*} = \begin{bmatrix} 
              \sqrt{\frac{q}{p}} \\
              0\\
             \frac{1+r \sqrt{\frac{q}{p}}}{\sqrt{{p}{q}}} 
    \end{bmatrix}, 
     {X}_{3}^{*} = \begin{bmatrix} 
              -\frac{1}{r} \\
              0 \\
              0 \\
    \end{bmatrix}, {X}_{4}^{*} = \begin{bmatrix} 
              - \sqrt{\frac{q}{p}} \\
              0 \\
               \frac{ -1+r \sqrt{\frac{q}{p}}}{\sqrt{{p}{q}}} \\
    \end{bmatrix}.
    \end{array}
    \end{equation}

From an ecological point of view, negative population density is not realistic as the population can not be negative, therefore, we take the vector ${\textbf{x}} = ({x}_{1},{x}_{2},{x}_{3})$ as an element of ${R}_{+}^{3}$.  ${R}_{+}^{3}$ is defined as
\begin{equation}
{R}_{+}^{3} = \{ {X} \in {R}^{3}:{x}_{i} \geq 0 ~ for ~ {i}\in \{1,2,3\}\}.
\end{equation}
Since equilibrium points  ${X}_{0}^{*},~ {X}_{1}^{*}$ and ${X}_{2}^{*}$ are elements of the set  $Int(R_{+}^{3})$, therefore, we study the local stability of ecologically feasible equilibrium points ${X}_{0}^{*},~ {X}_{1}^{*}$ and ${X}_{2}^{*}$.
 \begin{enumerate}[(I)]
 \item  Stability of Trivial Equilibrium Point ${X}_{0}^{*}$.
  
\begin{theorem} Consider the dynamics of the model (1) in the following  form
\begin{equation}
\dot{\textbf{x}} =H(\textbf{x})= A {\textbf{x}} +f({\textbf{x}})~ for~{\textbf{x}}\in {R}^{3}.
 \end{equation} 
If following three conditions are satisfied 
\begin{enumerate}[(i)]
\item Constant matrix ${A_{3 \times 3}}$ has ${3}$ Eigen-values with non-zero real part,
\item   $f({\textbf{x}})$ is smooth and 
\item  $\lim_{||\textbf{x}|| \to 0 }\frac{||f(\textbf{x})||}{||\textbf{x}||} = 0$,
\end{enumerate} 
then in a neighbourhood of the critical point ${X}_{0}^{*} = (0,0,0)$, there exists stable and unstable manifolds $W_{s}$ and $W_{u}$ with the same dimensions ${n}_{s}$ and ${n}_{u}$ as the stable and unstable manifolds ${E}_{s}$ and ${E}_{u}$ of the system
                  $$ \dot{{Z}}(t) = A{Z}.$$
In $\textbf{x} = \textbf{0}$, ${E}_{s}$ and ${E}_{u}$ are tangent to ${W}_{s}$ and ${W}_{u}$\cite{simmons2016differential}.
\end{theorem} \label{thm}
\begin{proof}
For GLV system, ${X}_{0}^{*}={(0,0,0)}$ is trivial equilibrium point. Here we check all three mentioned conditions of theorem \ref{thm}.\\
\begin{enumerate}[(i)]
\item  Note that for model (1), the constant matrix $A_{{3}\times{3}}$ is
$$A=\begin{pmatrix}
                   1&0&0\\
                   0&-1&0\\
                   0&0&-q\\
    \end{pmatrix}$$\\
 The determinant of the matrix ${A}$ is non zero if ${q} \neq {0}$. Since, we have taken ${q}={3}$, therefore, all Eigen values of ${A}$ have non zero real part.
\item  Since functions $f_{1}(x_{1},x_{2},x_{3})$, $f_{2}(x_{1},x_{2},x_{3})$ and $ f_{3}(x_1,x_2,x_3)$ for model (1) are considered as
      \[ 
\begin{array}{ll}
     {x}_{1}(-{x}_{2}+{r}{x}_{1}-{p}{x}_{3}{x}_{1}) &= {f}_{1}({x}_{1},{x}_{2},{x}_{3}),\\
    {x}_{2}{x}_{1} & = {f}_{2}({x}_{1},{x}_{2},{x}_{3}), \\ 
     {x}_{3}({p}{x}_{1}^{2}) &     = {f}_{3}({x}_{1},{x}_{2},{x}_{3}).\\
\end{array}    
 \]
All three functions are continuous and have continuous partial derivative for all $\textbf{x}\in R^{3}$ which implies that $f(\textbf{x})$ is smooth on ${R}^{3}$. Hence, the second condition also holds.
\item For any ${\textbf{x}} \in {R}^{3}$, converting the Cartesian coordinates into spherical coordinates by making the following transformation
$$\begin{bmatrix} 
              {x}_{1} = {r} \sin{\theta} \cos{\phi}\\
              {x}_{2} = {r} \sin{\theta} \sin{\phi}\\
              {x}_{3} = {r} \cos{\theta} ~~~~~~~~~~~\\ 
    \end{bmatrix}, $$
where ${r}\geq {0}, {0} \leq {\theta} \leq {\pi}$ and ${0} \leq {\phi} \leq {\pi}.$\\
Using this transformation in model (1), we have   
$$\lim_{||{\textbf{x}}|| \to {0} }\frac{||f({\textbf{x}})||}{||\textbf{x}||} = 0.$$
Hence, the critical point $(0,0,0)$ of system $(1)$ is of the same type of critical point of the system 
                  $$ \dot{Z}(t) = A{Z}.$$
The Eigen values of ${A}$ are ${1},{-1}$ and $-{q}$ which implies that $(0,0,0)$ is saddle node for the  system $ \dot{Z}(t) = A{Z}$. Therefore, the trivial steady state $X_{0}^{*} = (0,0,0)$ of the model $(1)$ is a saddle point.
\end{enumerate}
 \end{proof}  
\item Stability of Axial Equilibrium Point ${X}_{1}^{*}$.
\\
 The Jacobian matrix of model (1) for parameter values $ p= 2.9851,~q= 3,~r=2$ is given as
 \begin{equation}
J=\begin{bmatrix}
1-{x}_2+4{x}_{1}-5.9702{x}_{1}x_{3} & -{x}_{1}  & -2.9851{x}_{1}^{2}\\
                               {x}_{2}        &-1+{x}_{1} &  0\\
          5.9702{x}_{1}{x}_{3}&    0       &-3+2.9851{x}_{1}^{2}\\
    \end{bmatrix}.
    \end{equation}
Jacobian matrix (31) of the model (1) about ${X}_{1}^{*}= (1,3,0)$ yields the following Jacobian matrix
\begin{equation}
J_{{X}_{1}^{*}}=\begin{bmatrix}
                {2} & {-1}  & {-2.9851}\\
                {3} &  {0}  &  {0}\\
                {0} & {0}  &   {-.0149}\\
    \end{bmatrix}.
    \end{equation}
 The characteristic equation $|J_{{X}_{1}^{*}}- \lambda I| = 0$ of  matrix $(32)$ is given as 
  \begin{equation}
      {\lambda}^{3} - (1.9851){\lambda}^{2} +(2.9702 ){\lambda} + 0.0447 =0.
  \end{equation} 
The characteristic equation $(33)$ has the following Eigen values
\begin{equation}
\lambda_{1}= -0.014900,~\lambda_{2}= 1+\sqrt{2}\iota,~\lambda_{3}= 1-\sqrt{2}\iota.
\end{equation}
 Since $\lambda_{2}$ and $\lambda_{3}$ have positive real parts, it implies that ${X}_{1}^{*} = (1,3,0)$ is unstable equilibrium point.
 \item Stability of Planer Equilibrium Point ${X}_{2}^{*}$.
\\
 The Jacobian matrix of model (1) for parameter values $ p= 2.9851,~q= 3,~r=2$ is given as
 \begin{equation}
 J=\begin{bmatrix}
1-{x}_2+4{x}_{1}-5.9702{x}_{1}x_{3} & -{x}_{1}  & -2.9851{x}_{1}^{2}\\
                               {x}_{2}        &-1+{x}_{1} &  0\\
          5.9702{x}_{1}{x}_{3}&    0       &-3+2.9851{x}_{1}^{2}\\
    \end{bmatrix}.
 \end{equation}
Jacobian matrix $(35)$ of the model (1) about ${X}_{2}^{*}= (1.002493,0,1.4159)$ yields the following Jacobian matrix
    \begin{equation}
    J_{{X}_{2}^{*}}=\begin{bmatrix}
                3.004986 & -1.002493  &  3.000002\\
                0        &  .002493  &  0      \\
                6.00887         &    0      &   .000002    \\
    \end{bmatrix}.
        \end{equation}
 The characteristic equation $|J_{{X}_{2}^{*}}- \lambda I| = 0$ of  matrix $(36)$ is given as
 \begin{equation}
    \lambda^{3} +(0.997507)\lambda^{2} +( 18.027423)\lambda - 0.044949 =0. 
 \end{equation}
 The characteristic equation $(37)$ has the following Eigen values
 \begin{equation}
 \lambda_{1}= -0.002493,~ \lambda_{2}= -0.5+4.216\iota,~ \lambda_{3}= -0.5-4.216\iota. 
  \end{equation}
 Since all three Eigen-values of matrix $(36)$ have negative real parts, it shows that ${X}_{2}^{*}= (1002493,0,1.4159)$ is locally stable equilibrium point. 
 \end{enumerate}
It is clear that planer equilibrium point is stable whereas trivial and axial equilibrium points are unstable equilibrium points. Since trivial equilibrium point refers the zero density of all three species, therefore, we neglect the instability of trivial equilibrium point. From an ecological point of view, we mainly focus on non-trivial unstable equilibrium point ${X}_{1}^{*}$ and try to stabilize it by adding some external control inputs. 
 \subsection{\label{sec3.7} Control of instability of  axial equilibrium point}
 In order to suppress instability to ${X}_{1}^{*} = (1,3,0)$, we consider the controlled GLV system in the following form
 \begin{equation}
\begin{array}{ll}
    \dot{{x}_{1}} &= {x}_{1}(1-{x}_{2}+{r}{x}_{1}-{p}{x}_{3}{x}_{1})+{u}_{1},\\
    \dot{{x}_{2}} &= {x}_{2}(-1 +{x}_{1})+{u}_{2}, \\ 
    \dot{{x}_{3}} &= {x}_{3}(-{q}+{p}{x}_{1}^{2})+{u}_{3}. \end{array} 
\end{equation}
 We introduce the external control law 
 \begin{equation}
\begin{array}{ll}
          {u}_{1} & = -{\mu}_{1}({x}_{1}-1),\\
   {u}_{2} &= -{\mu}_{2}({x}_{2}-3), \\ 
  {u}_{3} &= -{\mu}_{3}({x}_{3}-0). 
\end{array} 
\end{equation}
with ${x}_{1}$,~${x}_{2}$, ${x}_{3}$ as the feedback variable and $\mu_{1},~\mu_{2},~\mu_{3}$ as the positive feedback gains. We substitute control law $(40)$ into $(39)$ and hence, the controlled system $(39)$ takes the following form
 \begin{equation}
     \begin{array}{ll}
    \dot{{x}_{1}}= {x}_{1}(1-{x}_{2}+{r}{x}_{1}-{p}{x}_{3}{x}_{1}))-{\mu}_{1}({x}_{1}-1),\\
    \dot{{x}_{2}}= {x}_{2}(-1 +{x}_{1})-{\mu}_{2}({x}_{2}-3), \\ 
    \dot{{x}_{3}}= {x}_{3}(-{q}+{p}{x}_{1}^{2})-{\mu}_{3}({x}_{3}-0). 
\end{array}
\end{equation}
\begin{theorem}
 The equilibrium point ${X}_{1}^{*} = (1,3,0)$ of the model $(1)$ will be asymptotically stable if positive gains ${\mu}_{1}, {\mu}_{2}$ and ${\mu}_{3}$ satisfy the following inequalities\cite{yassen2006adaptive}
  \begin{equation}
     \begin{array}{ll}
  {\mu}_{1} & > 2,\\
  {\mu}_{1}{\mu}_{2} &> 1+2{\mu}_{2}, \\ 
    {\mu}_{1}{\mu}_{2}({\mu}_{3}+0.0149) &>{\mu}_{2}(2{\mu}_{3}+2.2528)+{\mu}_{3}+0.0149. 
\end{array}
\end{equation}
 \end{theorem}
 \begin{proof}
 The Jacobian matrix $J$ of the system $(41)$  is given by
 \begin{equation}
 \begin{bmatrix}
1-{x}_{2}+2{x}_{1}(r-p{x}_{3})-{\mu}_{1}& -{x}_{1}  & -p{x}_{1}^{2}\\
                               {x}_{2}&-1+{x}_{1}-{\mu}_{2} &  0\\
          2p{x}_{1}x_{3}& 0 &-q+p{x}_{1}^{2}-{\mu}_{3}
          \end{bmatrix}
 \end{equation}
 
 Let us consider that 
 \begin{equation}
 \begin{array}{ll}
{e}_{1} &= ({x}_{1}-1),\\ 
{e}_{2} &= ({x}_{2}-3),\\
{e}_{3} &= ({x}_{3}-0).
 \end{array}
 \end{equation}
  From $(44)$, we get the error system as
 \begin{equation}
\begin{array}{ll}
    \dot{{e}_{1}}= (2-{\mu}_{1}){e}_{1}-{e}_{2}-2.9851{e}_{3},\\
    \dot{{e}_{2}}=  3{e}_{1}-{\mu}_{2}{e}_{2}, \\ 
    \dot{{e}_{3}}= -(0.014900+{\mu}_{3}){e}_{3}. \\
\end{array}
\end{equation}
The system $(1)$ with constant and known parameters, will be stabilized to steady state ${X}_{1}^{*} = (1,3,0)$, if error system $(45)$ stabilized to $(0,0,0)$. 
\\
To study the stability of equilibrium point $(0,0,0)$ of error system, we consider the Lyapunov function $L({e}_{1},{e}_{2},{e}_{3})$ as:
 \begin{equation}
    L = \frac{1}{2}({e}_{1}^{2} + {e}_{2}^{2} + {e}_{3}^{2}). 
 \end{equation}
 The time derivative of $L$ in the neighbourhood of $(0,0,0)$ is given as
 \begin{equation}
 \dot{L}= (2-{\mu}_{1}){e}_{1}^{2}+2{e}_{1}{e}_{2}-{\mu}_{2}{e}_{2}^{2}-2.9851{e}_{1}{e}_{3}-(0.0149+{\mu}_{2}){e}_{3}^{2}.
 \end{equation}
The time derivative of Lyapunov function can be re-written in the following form
\begin{equation}
\dot{L}= \textbf{e}^{T}M\textbf{e}.
\end{equation}
where $\textbf{e}=(({x}_{1}-1),({x}_{2}-3),({x}_{3}-0))$ is the error vector in ${R}^{3}$, $\textbf{e}^{T}$ is the transpose of error vector $\textbf{e}$ and the matrix $M$ is $3 \times 3$ is given as
 $$M = \begin{bmatrix}
2-{\mu}_{1}&1&-1.49250\\
     1&-{\mu}_{2}&0\\
 -1.49250& 0 &-(0.0149+{\mu}_{3})\\
    \end{bmatrix}$$
According to Lyapunov stability theory, the equilibrium point $(0,0,0)$ of system $(45)$ will be asymptotically stable if $\dot{L} < 0 $. And $\dot{L} < 0 $ if  matrix ${M}$ will be negative definite. Considering this, we find that mentioned condition will be fulfilled if positive feedback gains $\mu_{1}$, and $\mu_{2}$ satisfy  the following inequalities,
  \begin{equation}
  \begin{array}{ll}
    {\mu}_{1} & > 2,\\
    {\mu}_{1}{\mu}_{2} &> 1+2{\mu}_{2},\\
    {\mu}_{1}{\mu}_{2}({\mu}_{3}+0.0149)&>{\mu}_{2}(2{\mu}_{3}+2.2528)+{\mu}_{3}+0.0149.
\end{array}
 \end{equation} 
\end{proof}
\section{\label{sec4}Complete Replacement Synchronization}
To investigate complete replacement synchronization techniques, we consider  two identical chaotic GLV systems having the same parameter but different initial conditions. Since the coupling between models is needed to maintain the synchronous state, we couple the states of both models with two controllers and drive the response system with prey species ${x}_{1}$. For this, we remove prey from response system, and drive its counterpart. Here, we can think of prey species ${x}_{1}$ as a driving variable for response system with an assumption that it is superfluous in the system of two coupled GLV models. This construction gives us a new five-dimensional drive-response system having drive and response variables as $({x}_{1d}, {x}_{2d}, {x}_{3d})$ and $({x}_{2r}$, ${x}_{3r})$ respectively. The coupled chaotic system with ${x}_{1d}$ drive configuration is as follows:
\begin{equation}
\left \{
\begin{array}{ll}

 \textbf{(Drive System)}\\
   \dot{x}_{1d} & = {x}_{1d}(1-{x}_{2d}+{r}{x}_{1d}-{p}{x}_{3d}{x}_{1d}),\\
    \dot{x}_{2d} & = {x}_{2d}(-1+{x}_{1d}), \\ 
    \dot{x}_{3d} & = {x}_{3d}(-{q}+{p}{x}_{1d}^2),\\
    \\
\textbf{( Response System)} \\
    \dot{x}_{2r} & = {x}_{2d}(-1 +{x}_{1d})+ {u}_{1}, \\ 
    \dot{x}_{3r} & = {x}_{3r}(-q+{p} {x}_{1d}^{2})+{u}_{2}. 
\end{array} \right
\} 
\end{equation}
 $${x}_{2d}(0)\neq {x}_{2r}(0) ~ and ~ {x}_{3d}(0) \neq {x}_{3r}(0).$$
\subsection{\label{sec4.1}Active control law for stability of synchronization manifold}

\begin{theorem}
The identical synchronization manifold $\Omega = [{x}_{2d} = {x}_{2r},~{x}_{3d} = {x}_{3r} ]$ is globally asymptotically stable for the coupling between drive and response system in equation $(24)$ for positive  ${\mu}_{1}$ and ${\mu}_{2}$, where ${\mu}_{1}$ and ${\mu}_{2}$ are large enough such that 
$$  {\mu}_{1}+1>{x}_{1d}, \quad \quad  {\mu}_{2}+q>p {{x}_{1d}^2}.$$
\end{theorem} 
\begin{proof} We consider drive-response system given by equations $(50)$ and add the uni-directional controllers to the response system through the linear positive constants ${\mu}_{1}$ and ${\mu}_{2}$.  We choose two controllers for response system as
 \begin{equation}
\begin{array}{ll}
    {u}_{1 } &= -{\mu}_{1}({x}_{2r}(t)-{x}_{2d}(t)), \\ 
   {u}_{2} & = -{\mu}_{2}({x}_{3r}(t)-{x}_{3d}(t)).
  \end{array} 
\end{equation}

Existence of all forms of identical synchronization in any  dynamical system (chaotic or not), are really manifestations of dynamical behaviour restricted to a flat hyper-plane in the phase space i.e. to say motion is continually confined to a hyper-plane which can be referred as synchronization manifold \cite{pecora1998master}. Therefore, we consider the identical synchronization manifold of the systems equation $(50)$ as 
  $$\Omega = [{x}_{2d} = {x}_{2r},{x}_{3d} = {x}_{3r} ].$$ 
Further, we consider the errors between states of drive and response systems of system $(50)$ as
 \begin{equation}
\begin{array}{ll}
   {e}_{2}(t) &={x}_{2r}(t)-{x}_{2d}(t), \\ 
    {e}_{3}(t)& ={x}_{3r}(t)-{x}_{3d}(t).
  \end{array} 
\end{equation} 
 The dynamics of error system is as follows
  \begin{equation}
\begin{array}{ll}
    \dot{{e}_{2}} &= {(-1-{\mu}_{1} + {x}_{1d})}{e}_{2}, \\ 
    \dot{{e}_{3}} &= (-q-{\mu}_{2}+p {{x}_{1d}^2}){e}_{3}.
  \end{array} 
\end{equation} 
where, we are now interested in the stability of origin. The Jacobian of right side of $(53)$ is given by 
\begin{equation} 
J({e}_{2},{e}_{3})=\begin{bmatrix}
                -1+\mu_{1}+ {x}_{1d} & 0\\
                 0  & -q+\mu_{2}+p {x}_{1d}^2 \\
                \end{bmatrix}\
                \end{equation} 
We treat the response system $({x}_{2d}(t),{x}_{3r}(t))$ as a separate system driven by ${x}_{1d}$, then the solutions of equation $(54)$ convey us about convergence and divergence of two initially nearby trajectories of $\{{x}_{2r}(t),{x}_{2d}(t)\}$ and $\{{x}_{3r}(t),{x}_{3d}(t)\}$. 
 \\
Next, We analyse the possibility of synchronization using the Lyapunov function construction method. We consider the Lyapunov function as
\begin{equation}
L({e}_{2},{e}_{3}) = \frac{1}{2}({e}_{2}^{2} + {e}_{3}^{2}).
\end{equation}
\begin{equation}
\frac{dL}{dt} = {e_{2}}{\dot{e_{2}}}+ {e_{3}}{\dot{e_{3}}}.\end{equation}
 Plugging  dynamics of errors $(53)$ into $(55)$, we get 
 \begin{equation}
 \frac{d{L}}{d{t}} = -[(\mu_{1}+1-{x}_{1d}){e}_{2}^{2}+ (\mu_{2}+q-p {x}_{1d}^{2}){e}_{3}^{2}],
 \end{equation}
 which will be strictly negative for following conditions on ${\mu}_{1}$ and ${\mu}_{2}$
 \begin{equation} 
\begin{array}{ll}
   \mu_{1}+1> {x}_{1d}, \\ 
 \mu_{2} + q> p {{x}_{1d}^2}.
  \end{array} 
 \end{equation}  
for all ${t}>{0}$. Condition $(58)$ ensures that we consider the bounded density of prey species, then we can bound the positive feedback gains. Thus, if ${\mu}_{1}$ and ${\mu}_{2}$ satisfy $(58)$, then it can be assured that $\frac{d{L}}{d{t}}< {0}$ for all ${t}>{0}$ or in other words, the complete replacement synchronization follows as ${e}_{2}  ~ and ~{e}_{3} \rightarrow 0 $ as $t \rightarrow \infty$. 
\end{proof}

\subsubsection{Numerical Simulation}
Since, a suitable  coupling can influence both frequency as well as chaotic amplitude, therefore, the states coincide (or nearby coincide) and regime of synchronization sets in. Thus, it is pre-arranged that the chosen coupling should  assist the coupled states in coincidence  without perturbing their chaotic rhythm. We numerically integrate the System $(50)$ and display results in figure \ref{figa}, where drive and response systems are shown to synchronize when considered positive feed-back gain are chosen as ${\mu_{1}}=0.000024, {\mu_{2}}=1.345$. 
\begin{figure}[ht]
\centering
\begin{tabular}{cc}
 \subfigure[]{\includegraphics[width=7cm,height=5cm]{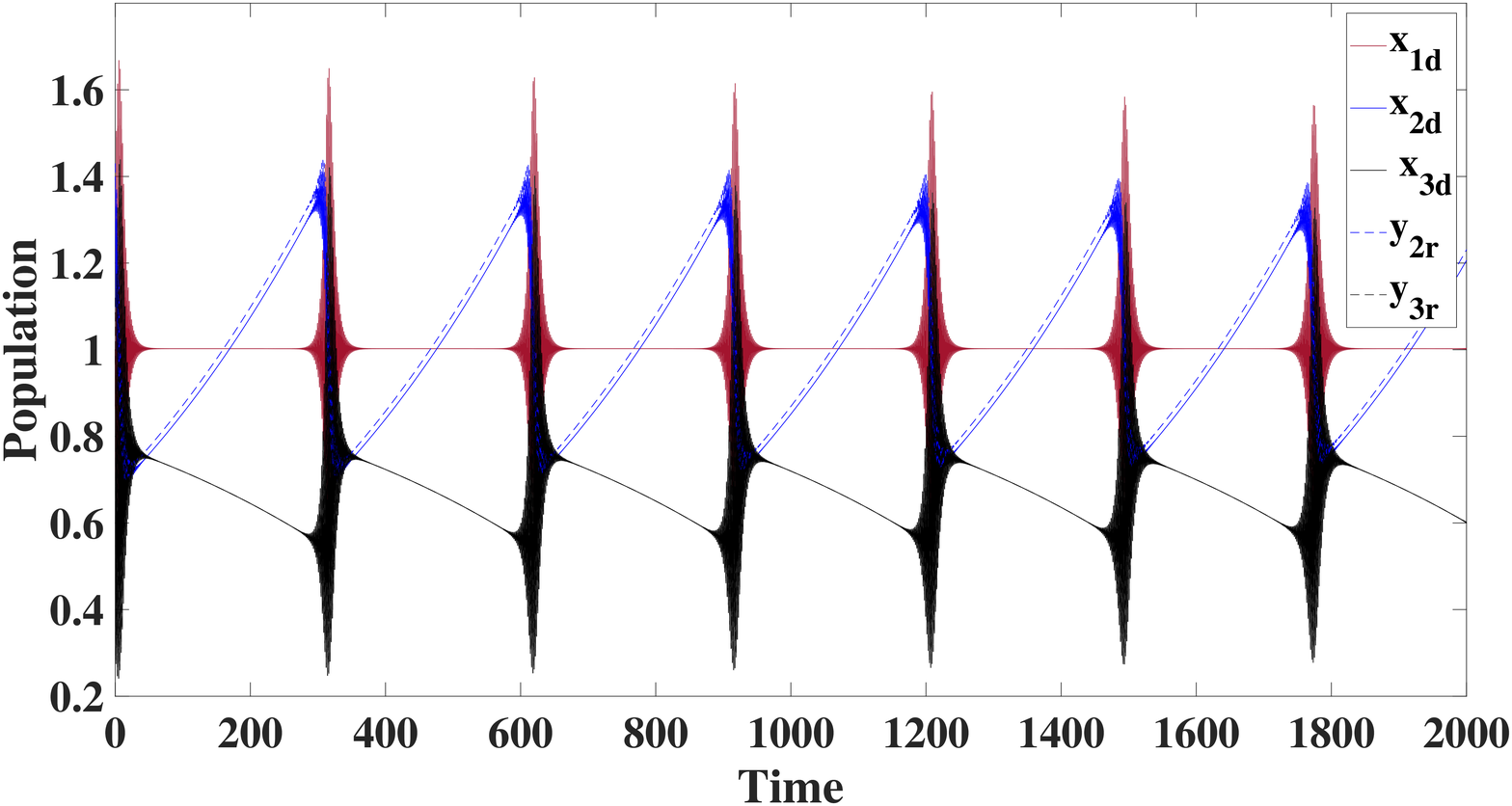}}  &  
  \subfigure[]{\includegraphics[width=7cm,height=5cm]{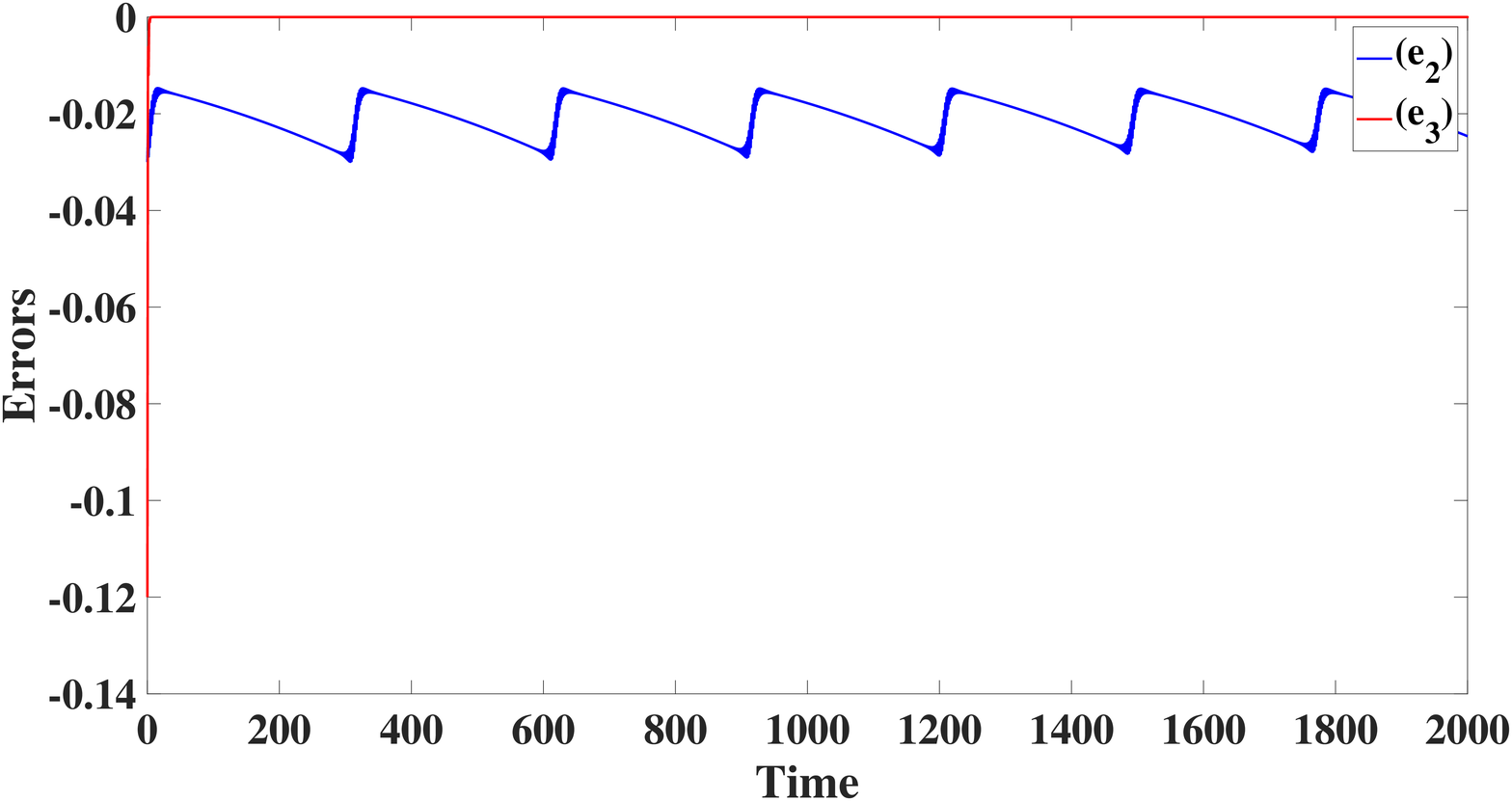}}    \\
  \subfigure[]{\includegraphics[width=7cm,height=5cm]{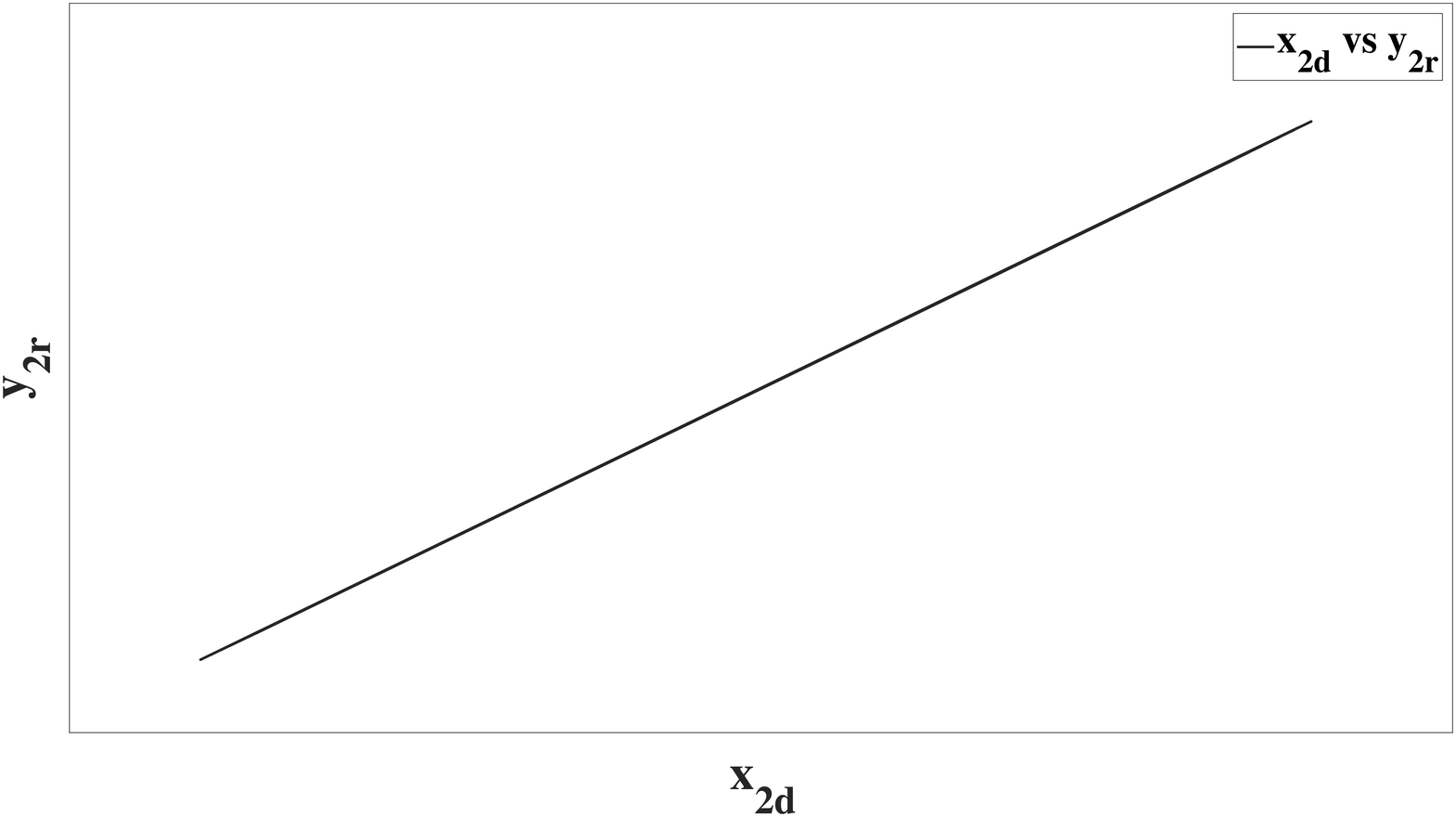}}  &
  \subfigure[]{\includegraphics[width=7cm,height=5cm]{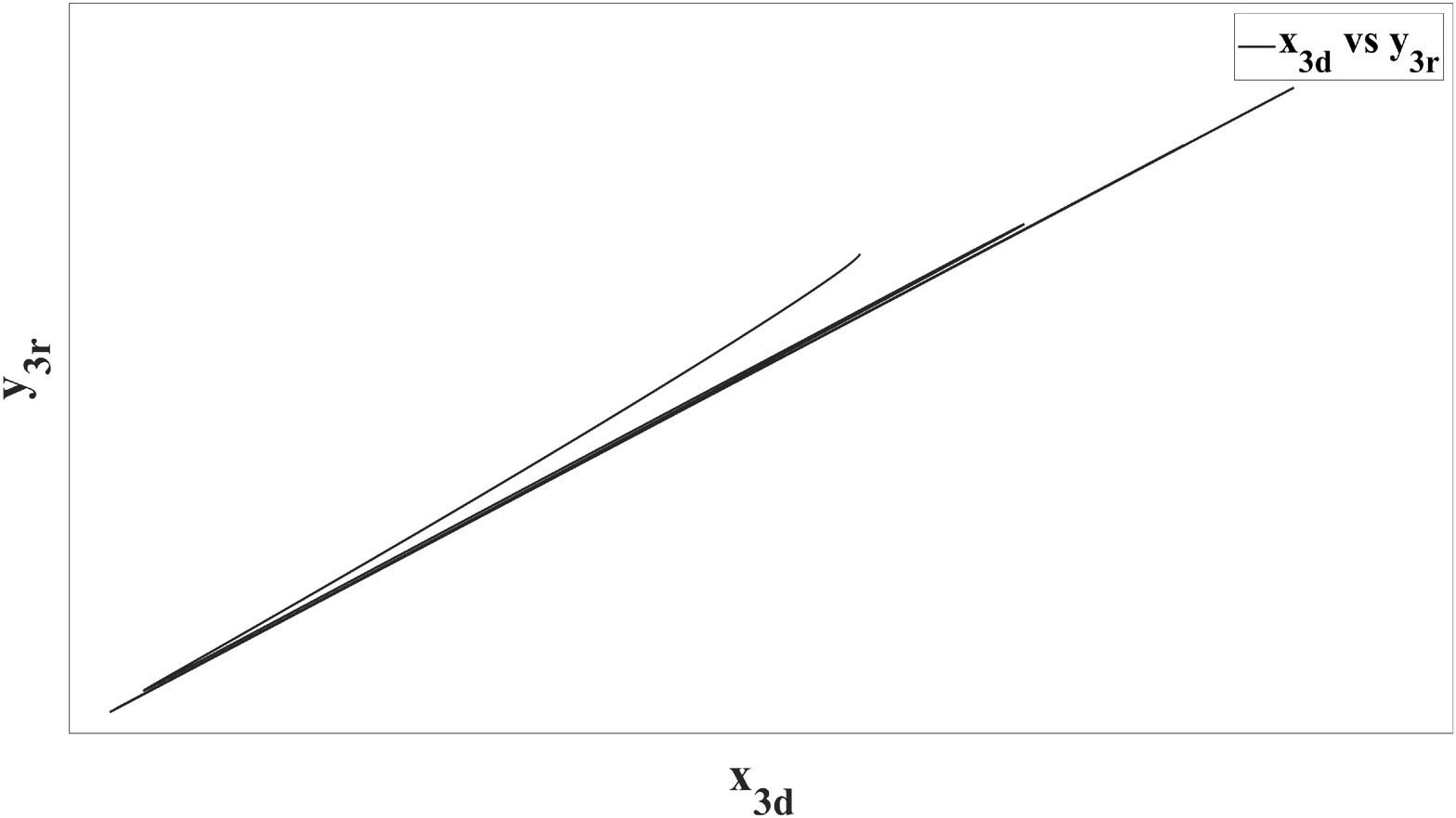}}  
\end{tabular}
\caption{(a):  solutions of drive and response systems plotted over time, (b): errors between drive and response systems  over time,(c): synchronization plot of $\{{x}_{2d},{x}_{2r}\}$, (d): synchronization plot of $\{{x}_{3d},{x}_{3r}\}$ } 
\end{figure} \label{figa}

\subsubsection{Lyapunov Spectrum For Drive-Response System}
Since the necessary condition for the stability of the synchronization manifold is the negative largest transverse Lyapunov exponent. In the case of complete replacement synchronization, the transverse Lyapunov exponents are also known as conditional Lyapunov exponents. It is because Lyapunov exponents for the new system depend on the coupling from the drive{\cite{frisk2016synchronization}}. The Lyapunov spectrum is a global indicator of the system's state, aggregating over the behaviour of the entire system trajectory in phase space. The typical approach for observing these transitions is to see a change in sign of the Lyapunov exponents of the system, as obtained from ensemble averages of the eigenvalues of the Jacobian matrix of system $(50)$ \cite{lahav2018synchronization}. 
We solve equations of system $(50)$ and get the Jacobian matrix as
\begin{equation}
J({x}_{1d},{x}_{2d},{x}_{3d},{x}_{2d},{x}_{3r})= {A}_{5 \times 5}, ~where~ A = [a_{{i}{j}}].
\end{equation}
The entries of the matrix $A$ are given as:
\begin{equation*}
\begin{split}
 {a}_{{1}{1}}={1-{x}_{2d}+2(r-p{x}_{3d}){x}_{1d}},\\  {a}_{{1}{2}}=-{x}_{1d},\quad {a}_{{1}{3}}={-p{x}_{1d}^{2}},\\ 
  {a}_{{2}{1}} ={x}_{2d},\quad {a}_{{2}{2}}={-1+{x}_{1d}},\\
  {a}_{{3}{1}}=2p{x}_{1d}{x}_{3d}, \quad {a}_{{3}{3}}={-q+p {x}_{1d}^{2}},
 \\ 
 {a}_{{4}{1}}={{x}_{2d}},\quad {a}_{{4}{2}}={\mu_{1}}, \quad  {a}_{{4}{4}}={-1+ {x}_{1d}-\mu_{1}{x}_{2d}},\\
 {a}_{{5}{1}}=2p {x}_{3r} {x}_{1d} , \quad {a}_{{5}{3}}=\mu_{2}, \quad {a}_{{5}{5}}=-q+p{x}_{1d}^{2}-\mu_{2},\\
 {a}_{{1}{4}}= \quad a_{{1}{5}}= \quad {a}_{{2}{3}} = \quad {a}_{{2}{4}}= \quad {a}_{{2}{5}}= \quad {a}_{{3}{2}}=0,\\
 {a}_{{3}{4}}= \quad {a}_{{3}{5}}= \quad 
 {a}_{{4}{3}}= \quad {a}_{{4}{5}}= \quad {a}_{{5}{2}}= \quad {a}_{{5}{4}}={0}.
\end{split}
\end{equation*}
Averaging the eigenvalues of Jacobian $J$  over all phase space configurations set-up by the chaotic trajectory, we get the five Lyapunov exponents of the system of two coupled GLV models.
 \begin{figure}[ht]
\begin{center}
\includegraphics[width=10cm,height=6cm]{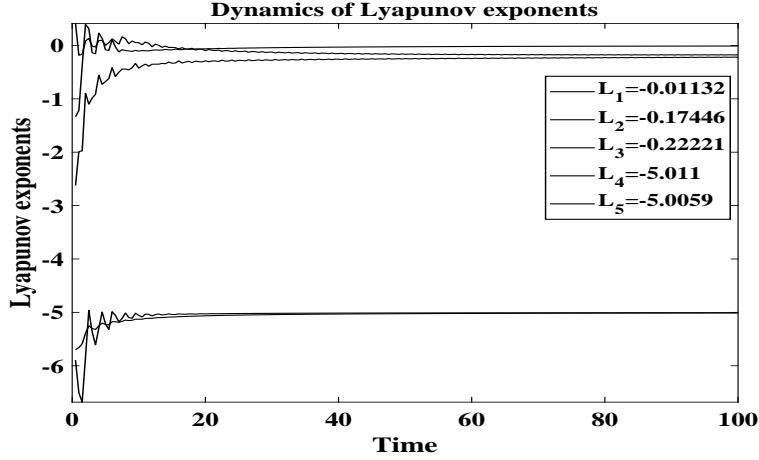}
\caption{Lyapunov Exponents of two GLV models coupled with positive feedback gains $\mu_{1}=0.000024$ and $\mu_{2}= 1.345$.}
\end{center}
\label{fig:8} 
\end{figure}
\\
For set of parameter values $\{p,q,r\} =\{2.9851,3,2\}$,  Lyapunov exponents of the drive-response system are obtained as
\begin{equation}
\begin{array}{ll}
{L}_{1}=-0.011320,\\
{L}_{2}=-0.174464, \\
 {L}_{3}=-0.22221,\\
 {L}_{4}= -5.011,\\
  {L}_{5}=-5.0059.
\end{array}
\end{equation}
Since, all Lyapunov exponents are negative which confirms the stable synchronization manifold, therefore, it can be concluded that the states of coupled GLV systems are synchronized. 
 
 \subsection{\label{sec:level7}Adaptive control law for stability of synchronization manifold}
Using the method \cite{fotsin2005adaptive}, we design non-linear adaptive controller for global complete-replacement synchronization of two chaotic GLV systems with unknown parameters. We consider the drive system as 
 \begin{equation}
\begin{array}{ll}
   \dot{x}_{1d} &= {x}_{1d}(1-{x}_{2d}+r {x}_{1d}-p {x}_{3d}{x}_{1d}),\\
    \dot{x}_{2d} &= {x}_{2d}(-1 +{x}_{1d}), \\ 
    \dot{x}_{3d} &= {x}_{3d}(-{q}+p {x}_{1d}^2).
  \end{array} 
 \end{equation} 
 The response system is  given by controlled  chaotic system
 \begin{equation}
\begin{array}{ll}
    \dot{x}_{2r} &= {x}_{2r}(-1 +{x}_{1d}) + {u}_{1}, \\ 
    \dot{x}_{3r} &= {x}_{3r}(-{q}+{p} {x}_{1d}^2)+{u}_{2}. 
\end{array} 
\end{equation} 
 The synchronization error between drive and response systems is defined as
\begin{equation}
\begin{array}{ll}
    {e}_{2}(t)& ={x}_{2r}(t)-{x}_{2d}(t) \\ 
   {e}_{3}(t) &= {x}_{3r}(t)-{x}_{3d}(t). 
\end{array} 
\end{equation} 
 The error dynamics  between drive and response systems is calculated as :
\begin{equation}
\begin{array}{ll}
    \dot{{e}_{2}} &= -{e}_{2} + {x}_{1d}{e}_{2} + {u}_{1}, \\ 
    \dot{{e}_{3}} &= -{q} {e}_{3}+{p} {x}_{1d}^{2} {e}_{3} + {u}_{2}.\\ 
  \end{array} 
\end{equation} 
In $(64)$, unknown parameters ${p}$ and  ${q}$  are to be determined by using parameter estimates ${P(t)}$ and ${Q(t)}$ respectively. For this purpose, we consider  adaptive control laws ${u}_{1}$ and ${u}_{2}$ with positive feedback gains  ${\mu}_{1}$ and ${\mu}_{2}$ as
\begin{equation}  
\begin{array}{ll}
   {u}_{1} &={e}_{2} - {x}_{1d}e_{2} - {\mu}_{1}e_{2}, \\ 
   {u}_{2} &= Q(t) {e}_{3}- P(t) {x}_{1d}^{2} {e}_{3} - {\mu}_{2}{e}_{3}.
  \end{array} 
\end{equation} 
Using control law $(65)$ into the error dynamics $(64)$, we get
\begin{equation}  
\begin{array}{ll}
    \dot{{e}_{2}} &= -{\mu}_{1}{e}_{2}, \\ 
    \dot{{e}_{3}} &= -({q}-Q(t)) {e}_{3}+({p}-P(t)) {x}_{1d}^{2} {e}_{3} -{\mu}_{2}{e}_{3}.\\ 
  \end{array} 
\end{equation}   
We define parameter estimation error as
 \begin{equation} 
 \begin{array}{ll}
  {e}_{p} (t) = (p-P(t)), \\
  {e}_{q} (t) = (q-Q(t)).
\end{array}  
 \end{equation} 
 Using $(67)$, we can simplify the error dynamics $(66)$ as
  \begin{equation}  
\begin{array}{ll}
    \dot{{e}_{2}}= -{\mu}_{1}{e}_{2}, \\ 
    \dot{{e}_{3}}= - {e}_{q} (t) {e}_{3}+{e}_{p} (t) {x}_{1d}^{2} {e}_{3}- \mu_{2}{e}_{3}.\\ 
  \end{array} 
 \end{equation} 
Differentiating $(67)$ with respect to ${t}$, we get  
\begin{equation}  
\begin{array}{ll}
    \dot{{e}_{p}} &=-\dot{P}(t) , \\ 
    \dot{{e}_{q}} &= -\dot{Q}(t)
  \end{array} 
\end{equation} 
 
 \begin{theorem}  The identical synchronization manifold $\Omega = [{x}_{2d} = {x}_{2r},~{x}_{3d} = {x}_{3r} ]$ is globally asymptotically stable for the coupling between derive and response system in equation $(50)$ for positive  ${\mu}_{1}$ and ${\mu}_{2}$.
 \end{theorem} 
 \begin{proof} The identical synchronization manifold for the systems equation $(24)$ can be written as 
  $$\Omega = [{x}_{2d} = {x}_{2r},{x}_{3d} = {x}_{3r} ].$$
   Using change of coordinates
 \begin{equation}
 \textbf{e}=\begin{bmatrix}  {e}_{2} \\ {e}_{3}    \\ \end{bmatrix} = \begin{bmatrix}   {x}_{2d} - {x}_{2d}\\ {x}_{3r} - {x}_{3d}    \\ \end{bmatrix}, 
  \end{equation}
 such that $\Omega$ can be written  
 $$\Omega = (0,0).$$
 Next, we use Lyapunov stability theory for finding an update law for the parameter estimates. we consider the quadratic Lyapunov function  as 
 \begin{equation}  
 L = \frac{1}{2}({e}_{2}^{2} + {e}_{3}^{2} + {e}_{p}^{2} + {e}_{q}^{2}) 
 \end{equation} 
 Note that Lyapunov function ${L}$ is positive definite on ${R}^{4}$. Differentiating ${L}$ along the trajectories of $(66)$ and $(69)$. We get,
 \begin{equation}
 \begin{array}{ll}
\frac{d{L}}{d{t}} &= {{e}_{2}}{\dot{{e}_{2}}}+ {{e}_{3}}{\dot{{e}_{3}}}+{{e}_{p}}{\dot{{e}_{p}}}+{{e}_{q}}{\dot{{e}_{q}}},\\
 \frac{d{L}}{d{t}} &=-{\mu}_{1}e_{2}^{2}-{\mu}_{2}e_{3}^{2}+ {e}_{p}(t)(-\dot{P}(t)+{x}_{1d}^{2} {e}_{3})+{e_{q}(t)}(-\dot{Q}(t)-{e}_{3}^{2}).
 \end{array}   
 \end{equation}
 We want error system to be asymptotically stable i.e.
  \begin{equation}  
  \frac{d{L}}{d{t}} < 0.
  \end{equation}
In view of $(72)$, we take the parameter update law as 
\begin{equation} 
\begin{array}{ll}
    \dot{P}(t) = {x}_{1d}^{2} {e}_{3}, \\ 
   \dot{Q}(t) = -{e}_{3}^{2}. \\ 
  \end{array} 
\end{equation}
  By substituting the parameter update law $(74)$ into Lyapunov function, we obtain time derivative of ${L}$ as 
   \begin{equation}
       \frac{d{L}}{d{t}} = -{\mu}_{1}{e}_{2}^{2}- {\mu}_{2}{e}_{3}^{2},
       \end{equation}
From $(75)$, it is clear that $\dot{L}$ is negative semi-definite  function on ${R}^{4}$. Thus, we can conclude that the synchronization error vector $\textbf{\textit{e(t)}}$ and the parameter estimation error are globally bounded, i.e.
   \begin{equation}
   [{e}_{2},e_{3},e_{p},e_{q}]^{T} \in \textbf{L}_{\infty}.
   \end{equation}
  We define $ \mu = min \{\mu_{1},\mu_{2}\}$, then it follows from $(75)$ that 
   \begin{equation}
  \frac{d{L}}{d{t}} \leq -\mu  ||\textbf{e}||^{2}.
   \end{equation}
  Integrating the inequality $(77)$ with respect to $\tau$ from $0$ to $t$. We get,
   \begin{equation}
   \int_{0}^{t} \mu  ||\textbf{e}(\tau)||^{2} d\tau \leq L(0)-L(t). 
    \end{equation}
  From $(78)$ it follows that $\textbf{e} \in \textbf{L}_{2}$ and hence, $\dot{\textbf{e}}(t)\in \textbf{L}_{\infty} $.
\\
 With the help of Barbalat's lemma \cite{min2007barbalat},\cite{sun2009barbalat}, we conclude that $\textbf{e}(t) \rightarrow {0}$ exponentially as $t \rightarrow \infty$ for all initial conditions $\textbf{e}(0) \in R^{2}$.
  \end{proof}  
 
   \subsubsection{Numerical Simulation}
For numerical simulations, we use the classical fourth-order Runge-Kutta method to solve the GLV system.  The initial value of parameter estimates are taken as $P(0)=3.9,~ Q(0)=4$. The initial values of states of drive and response systems are taken as $({x}_{1d}(0),{x}_{2d}(0),{x}_{3d}(0))=(4, 1.4, 1.41)$ and $({x}_{2r}(0)$,  ${x}_{3r}(0))=(1,1.414)$ respectively. The effectiveness of control law is verified through simulation results which are shown in figure \ref{adap}. Figure  shows the solutions of system $(50)$.  It is clear that although the initial values are different, the error dynamics approach to zero as time goes to $\infty$. Therefore, our numerical results confirm that the amplitude and frequency of state variables of response system become same with the drive system under the designed control law. 
   \begin{figure}[ht]
\centering
\begin{tabular}{cc}
 \subfigure[]{\includegraphics[width=7cm,height=5cm]{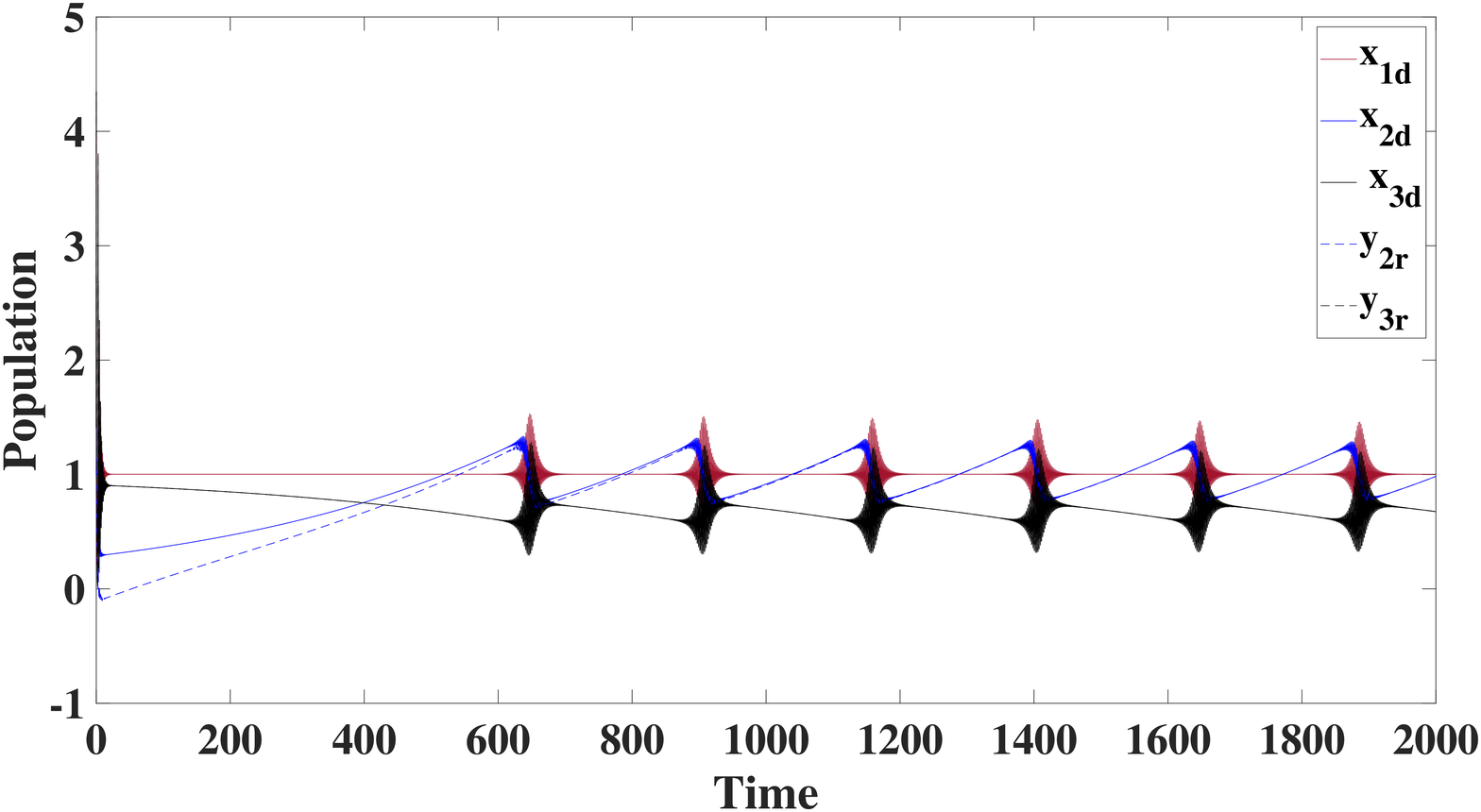}}  &  
  \subfigure[]{\includegraphics[width=7cm,height=5cm]{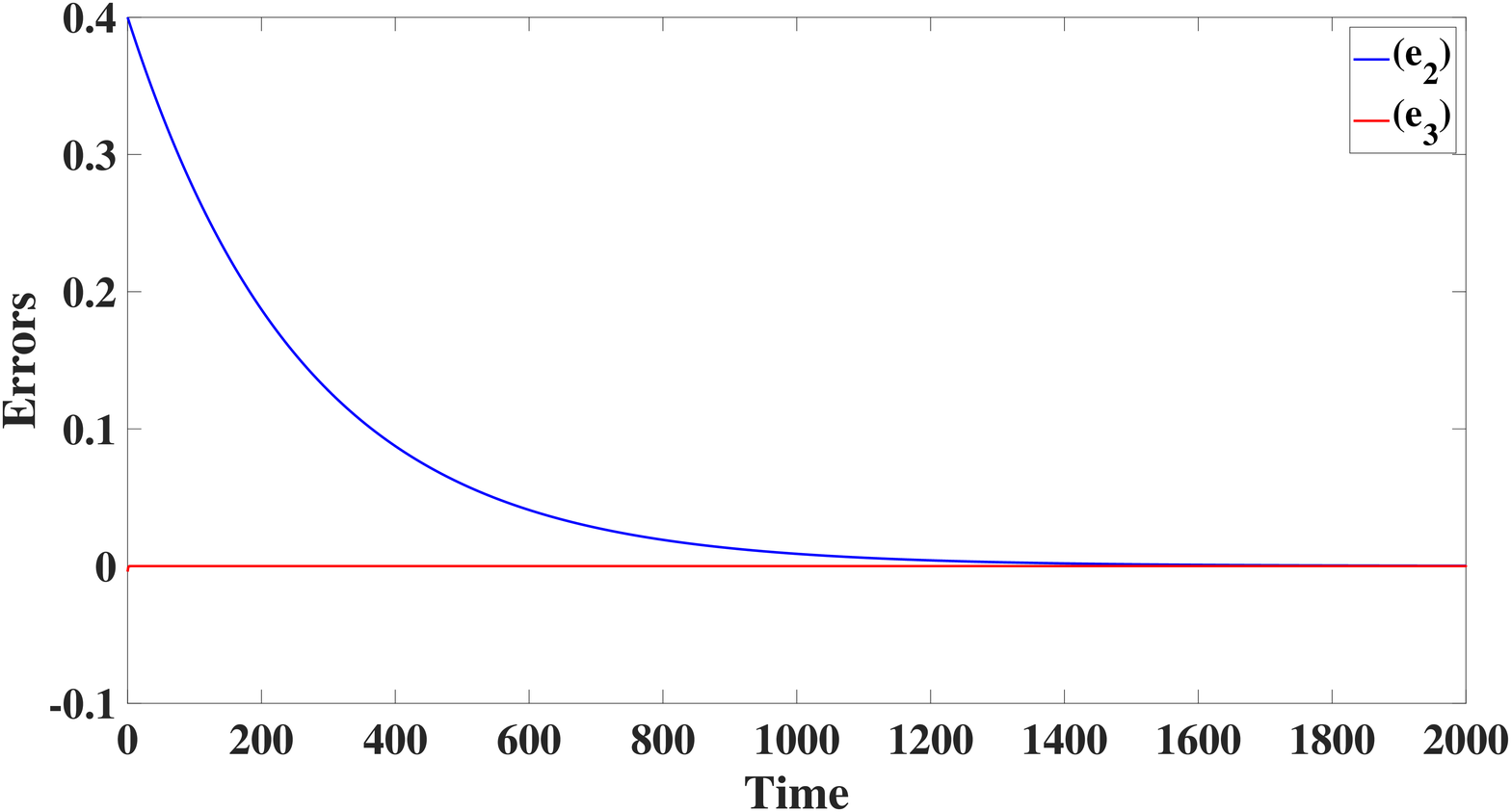}}  \\ 
  \subfigure[]{\includegraphics[width=7cm,height=5cm]{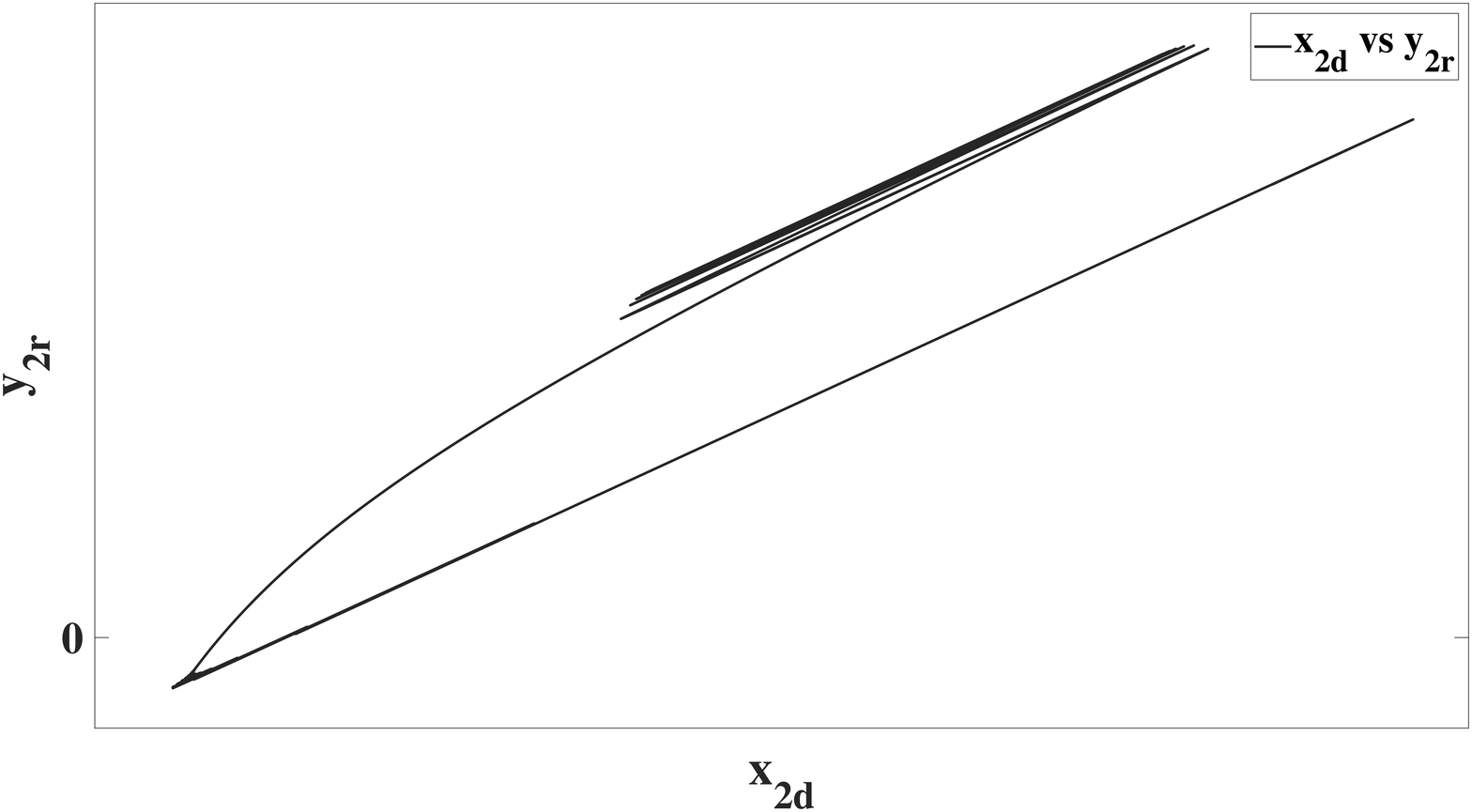}}  & 
  \subfigure[]{\includegraphics[width=7cm,height=5cm]{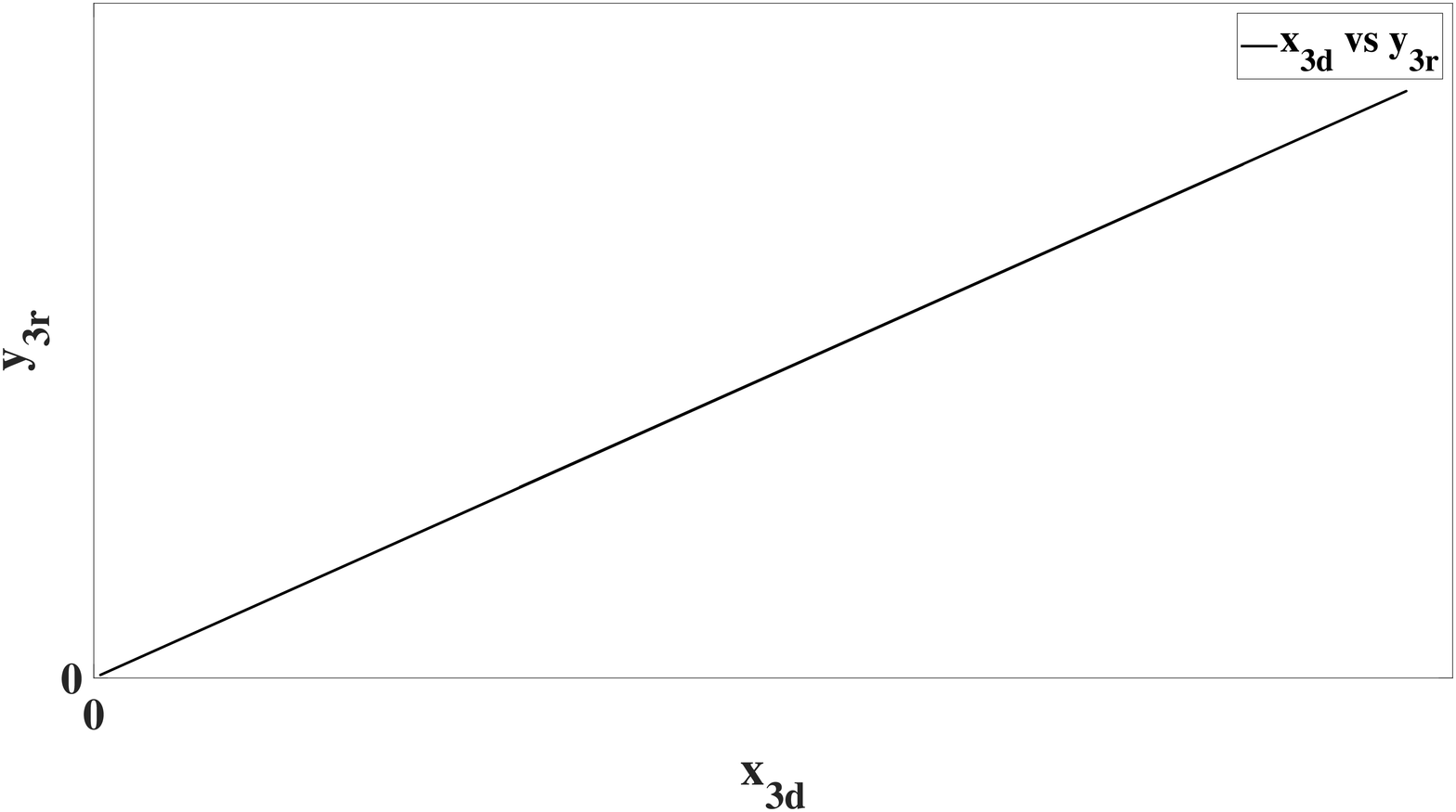}}
\end{tabular}
\caption{(a): solutions of drive and response systems over time, (b): errors between drive and response systems over time, (c): synchronization plot for $\{{x}_{2d}, {x}_{2r}\}$, (d): synchronization plot for $\{{x}_{3d},{x}_{3r}\}$ for coupling strengths  $\mu_{1}=0.0038$ and ${\mu}_{2}=2$.} 

\end{figure} \label{adap}
\section{\label{sec5}Conclusion}
This work examines predator-prey systems in ecosystems to understand how they contribute to sustainable environment. Our focus is on the use of Generalized Lotka-Volterra (GLV) equations to model the competition and trophic relationships between various species. We consider three different forms of three-dimensional GLV models, each with different functional responses (linear, Holling type II, and Holling type III). We find that the model with the linear functional response exhibits unstable dynamics, where alteration in functional response can stabilize the system dynamics for a particular scenario. To stabilize the dynamics in patchy ecosystem, we focus on the GLV model with the linear functional response for the remainder of the study. We investigate its fundamental properties and also examine the stability of equilibrium points and the suppression of instability at equilibrium. Through computation of Lyapunov exponent, we find that the model is chaotic due to one positive Lyapunov exponent and has two unstable equilibrium points for the constant parameters $p=2.9851, q=3, r=2$.
\\
Further, we investigate the synchronization of two chaotic GLV models using two control schemes: the Active Control Technique and the Adaptive Control Technique. We consider a configuration in which the prey population in the drive system acts as a driving variable for the response system, allowing the other two predator populations to depend only on the prey population. Using the Active Control Technique, we apply two simple linear controllers to synchronize the states of the GLV systems. These controllers are easy to implement and more straightforward than previous results. The stability of synchronization manifold is ensured through the transition of positive conditional Lyapunov exponent to negative one.  We also examine the synchronization of two chaotic GLV systems with unknown parameters using the Adaptive Control Technique. We design two adaptive laws of parameters using the Lyapunov stability theory to ensure global and exponential synchronization of the systems. Our results show that both the Active and Adaptive Control Techniques are effective for achieving global synchronization in chaotic systems.  
\section*{Funding}
The second author's research was funded by the Science and Engineering Research Board (SERB), under two separate grants with grant numbers MTR/2018/000727 and EMR/2017/005203.

\section*{Disclosure statement}
The authors declare that they have no conflict of interest.

\bibliographystyle{tfnlm}
\bibliography{aipsamp}

\end{document}